%% file: SIAM-SBP-quadratic-state-cost.tex
\documentclass[review,hidelinks,onefignum,onetabnum]{siamart220329}

\input{ex_shared}

\ifpdf
\hypersetup{
  pdftitle={Schr\"{o}dinger Bridge with Quadratic State Cost is Exactly Solvable},
  pdfauthor={Alexis M.H. Teter, Wenqing Wang, and Abhishek Halder}
}
\fi

\begin{document}
\nolinenumbers
\maketitle

\begin{abstract}
Schr\"{o}dinger bridge is a diffusion process that steers a given distribution to another in a prescribed time while minimizing the effort to do so. It can be seen as the stochastic dynamical version of the optimal mass transport, and has growing applications in generative diffusion models and stochastic optimal control. {\black{We say a Schr\"{o}dinger bridge is ``exactly solvable'' if the associated uncontrolled Markov kernel is available in closed form, since then the bridge can be numerically computed using dynamic Sinkhorn recursion for arbitrary endpoint distributions with finite second moments.}} In this work, we propose a regularized variant of the Schr\"{o}dinger bridge with a quadratic state cost-to-go that incentivizes the optimal sample paths to stay close to a nominal level. 
  Unlike the conventional Schr\"{o}dinger bridge, the regularization induces a state-dependent rate of killing and creation of probability mass, and its solution requires determining the Markov kernel of a reaction-diffusion partial differential equation. We derive this Markov kernel in closed form, {\black{showing that the regularized Schr\"{o}dinger bridge is exactly solvable, even for non-Gaussian endpoints. This advances the state-of-the-art because closed form Markov kernel for the regularized Schr\"{o}dinger bridge is available in existing literature only for Gaussian endpoints}}. Our solution recovers the heat kernel in the vanishing regularization (i.e., diffusion without reaction) limit, thereby recovering the solution of the conventional Schr\"{o}dinger bridge {\black{as a special case}}. 
  We deduce properties of the new kernel and explain its connections with certain exactly solvable models in quantum mechanics.
\end{abstract}



\section{Introduction}\label{sec:Intro}
The \emph{Schr\"{o}dinger bridge (SB)} -- now well-known in  stochastic control \cite{follmer1988random,leonard2012schrodinger,chen2015optimal,chen2021stochastic,caluya2021wasserstein} and machine learning \cite{de2021diffusion,wang2021deep,liu2022deep,vargas2023bayesian,bunne2023schrodinger,shi2024diffusion} -- originated as a thought experiment by physicist Erwin Schr\"{o}dinger \cite{schrodinger1931umkehrung,schrodinger1932theorie}. Specifically, Schr\"{o}dinger considered a prior model where samples from a known initial probability distribution $\mu_0$ at time $t=t_0$ evolve under standard Wiener process $\bm{w}_t\in\mathbb{R}^{n}$ until $t=t_1$, i.e., an It\^{o} diffusion
\begin{align}
\differential\bm{x}_t = \sqrt{2}\:\differential\bm{w}_t, \quad \bm{x}_{0} := \bm{x}_t\left(t=t_0\right) \sim \mu_0, \quad t\in[t_0,t_1].
\label{SDEBrownian}
\end{align}
Supposing that $\bm{x}_1 := \bm{x}_t\left(t=t_1\right)\sim\mu_1$, where the observed probability distribution $\mu_1$ at $t=t_1$ does not match with that predicted by \eqref{SDEBrownian}, Schr\"{o}dinger sought to find the \emph{most-likely} mechanism to reconcile this \emph{distributional observation versus prior physics mismatch}.

One way to formalize the phrase ``most-likely'' is via the principle of large deviations \cite{dembo2009large} using Sanov's theorem \cite{sanov1957probability,csiszar1984sanov} which poses the SB as a \emph{constrained maximum likelihood problem} on the path space $\Omega := \mathcal{C}\left([t_0,t_1];\mathbb{R}^{n}\right)$ (the space of continuous curves in $\mathbb{R}^{n}$ parameterized by $t\in[t_0,t_1]$). Let $\mathcal{M}(\Omega)$ denote the collection of all probability measures on $\Omega$, and let $\Pi_{01}:=\{\mathbb{M}\in\mathcal{M}(\Omega)\mid \mathbb{M}$ has marginal $\mu_i$ at time $t=t_i \forall i\in\{0,1\}\}$. Then, the SB is identified with the law $\mathbb{P}\in\Pi_{01}$ that solves the stochastic calculus-of-variations problem:
\begin{align}
\underset{\mathbb{P}\in\Pi_{01}}{\arg\inf}\: D_{\rm{KL}}\left(\mathbb{P}\parallel\mathbb{W}\right),
\label{KLmin}
\end{align}
where $D_{\rm{KL}}\left(\cdot\parallel\mathbb{W}\right)$ is the relative entropy a.k.a. Kullback-Leibler divergence w.r.t. the Wiener measure $\mathbb{W}$. The existence-uniqueness for \eqref{KLmin}, and thus for the SB, are guaranteed assuming the endpoints $\mu_0,\mu_1$ have finite second moments; see e.g., \cite{fortet1940resolution,beurling1960automorphism,jamison1974reciprocal}.

A \emph{dynamical} reformulation of  ``most-likely'' that is mathematically equivalent \cite{wakolbinger1990schrodinger,dawson1990schrodinger,aebi1992large,leonard2011stochastic} to \eqref{KLmin}, results from considering the most parsimonious or minimum effort correction of prior physics, i.e., the It\^{o} SDE $\differential\bm{x}_t^{\bm{u}} = \bm{u}(t,\bm{x}_t^{\bm{u}}) \differential t + \sqrt{2}\differential\bm{w}_t$ where we interpret $\bm{u}$ as corrective drift (the prior physics is the case $\bm{u}\equiv\bm{0}$). Specifically, among all \emph{finite energy Markovian} $\bm{u}$ such that the process $\bm{x}_{t}^{\bm{u}}$ satisfies the endpoint constraints $\bm{x}_{0}^{\bm{u}}\sim\mu_0, \bm{x}_{1}^{\bm{u}}\sim\mu_1$, we seek the $\bm{u}$ that minimizes the corrective effort $\mathbb{E}_{\bm{x}_t^{\bm{u}}\sim\mu_t}\left[\int_{t_0}^{t_1}\frac{1}{2}\vert\bm{u}\vert^2\:{\rm{d}}t\right]$ where $\mathbb{E}_{\mu_t}\left[\cdot\right]:=\int_{\mathbb{R}^{n}}(\cdot)\differential\mu_t$ denotes expectation w.r.t. the probability distribution $\mu_t\equiv\mu_t(t,\bm{x}_t^{\bm{u}})$ $\forall t\in[t_0,t_1]$. This transcribes SB into a stochastic optimal control problem:
\begin{subequations}
\begin{align}
&\underset{(\mu_t,\bm{u})}{\arg\inf}\int_{\mathbb{R}^{n}}\int_{t_0}^{t_1}\frac{1}{2}\vert\bm{u}\vert^2\:{\rm{d}}t\:\differential\mu_t
\label{SBclassicalObj}\\
\text{subject to}\quad & \frac{\partial\mu_t}{\partial t} = -\nabla\cdot\left(\mu_t\bm{u}\right) + \Delta\mu_t,
\label{SBclassicalPDEConstr}\\
& \mu_t(t=t_0,\cdot) = \mu_0\:(\text{given}), \quad \mu_t(t=t_1,\cdot) = \mu_1\:(\text{given}),\label{SBclassicalEndpointConstr}
\end{align}
\label{SBclassical}
\end{subequations}
where $\nabla\cdot$ and $\Delta$ denote the Euclidean divergence and Laplacian operators w.r.t. $\bm{x}_t^{\bm{u}}$, respectively. Constraint \eqref{SBclassicalPDEConstr} is the controlled Fokker-Planck a.k.a. forward Kolmogorov PDE governing the evolution of the law of the process $\bm{x}_{t}^{\bm{u}}\sim\mu_t$. 

We refer to \eqref{SBclassical} as the \emph{classical SB} and note that when the Laplacian term in \eqref{SBclassicalPDEConstr} is zero, then \eqref{SBclassical} specializes to the \emph{optimal transport (OT)} problem \cite{benamou2000computational}. Thus, OT can be seen as the case of null prior physics with corrected physics being an ODE $\frac{\differential\bm{x}^{\bm{u}}}{\differential t} = \bm{u}(t,\bm{x}_t^{\bm{u}})$. That SB generalizes OT can also be understood from static optimization viewpoint by showing \cite{chen2016relation,gentil2017analogy} that SB is precisely entropy-regularized OT \cite{cuturi2013sinkhorn}.

In the minimizing tuple $(\mu_t^{\rm{opt}},\bm{u}^{\rm{opt}})$ for \eqref{SBclassical}, the $\bm{u}^{\rm{opt}}$ denotes the \emph{optimal control policy}. If the endpoints $\mu_0,\mu_1$ are absolutely continuous with respective PDFs $\rho_0,\rho_1$, then so is the optimal joint distribution $\mu_t^{\rm{opt}}$, i.e., $\mu_t^{\rm{opt}} = \rho_t^{\rm{opt}}(t,\bm{x}_t^{\bm{u}})\differential\bm{x}_t^{\bm{u}}$. 
Furthermore, the solution can be found in terms of the \emph{heat kernel} $\kappa_0$ in the sense that $\rho_t^{\rm{opt}}$ factorizes into the product of a $\kappa_{0}$-harmonic (i.e., a solution of the backward heat PDE) and a $\kappa_0$-coharmonic (i.e., a solution of the forward heat PDE).

\begin{figure}[htbp]
  \centering
  \includegraphics[width=0.65\textwidth]{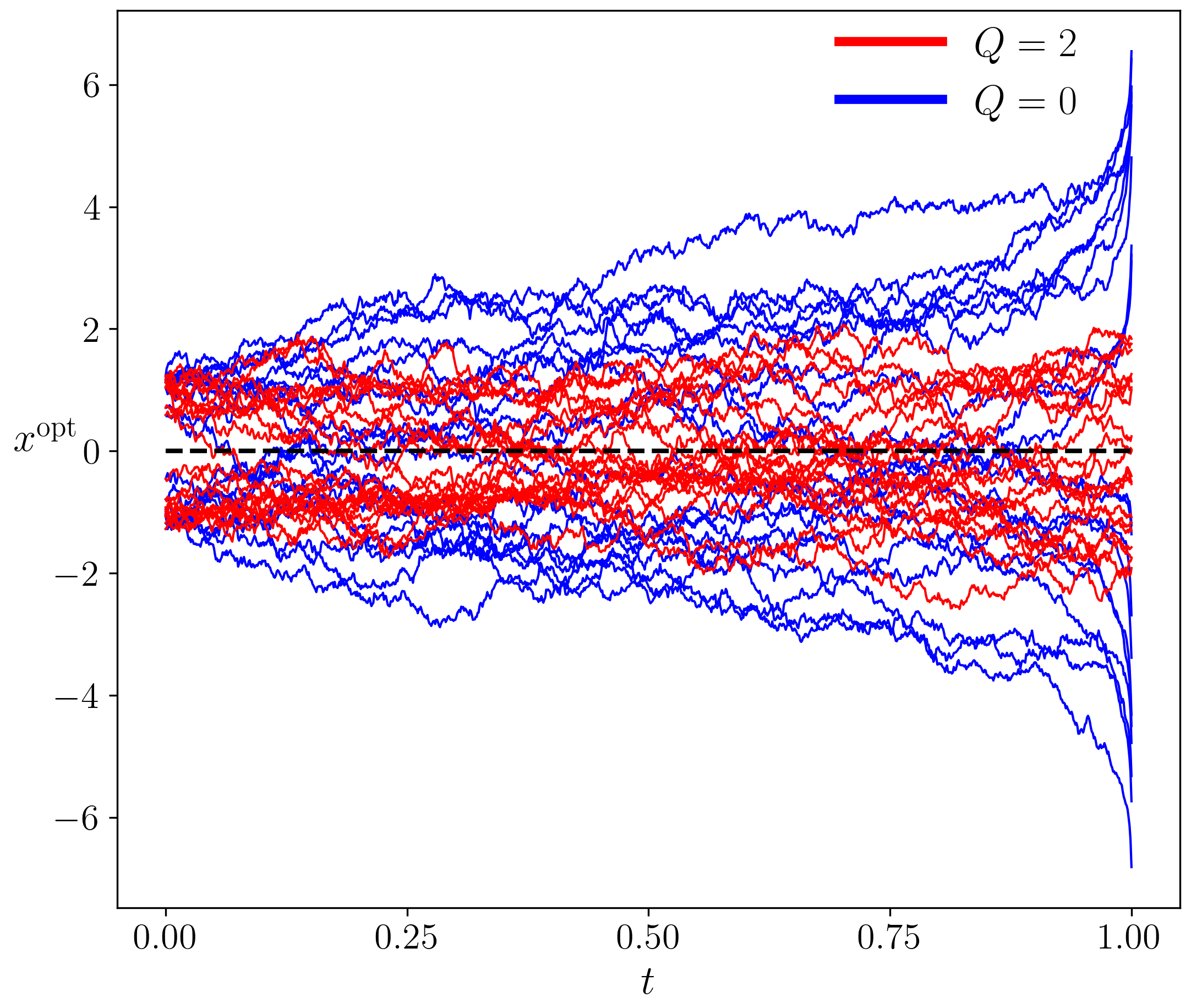}
  \caption{{\small{Sample paths for 1D SB with state cost $\frac{1}{2}Qx^2$ for fixed $Q\geq 0$, time horizon $[t_0,t_1]=[0,1]$, and endpoint PDFs $\rho_0 = \sum_{\mu_0\in\{-1,1\}}\frac{1}{2}\mathcal{N}(\mu_0,0.05^2)$, $\rho_1=\mathcal{N}(0,0.5^2)$.}}}
  \label{Fig:1dSBPsamplepathsWithAndWithoutq}
\end{figure}
This solution structure was discovered in \cite{schrodinger1931umkehrung,schrodinger1932theorie} and later found to hold in more general settings too, e.g., when prior physics is of the form\footnote{Here $\bm{f},\bm{g}$ are regular enough so that strong solutions exist for $\differential\bm{x}_t = \bm{f}(t,\bm{x}_t)\differential t + \sqrt{2}\bm{g}(t,\bm{x}_t)\differential\bm{w}_t$.}: $\differential\bm{x}_t = \bm{f}(t,\bm{x}_t)\differential t + \sqrt{2}\bm{g}(t,\bm{x}_t)\differential\bm{w}_t$ generalizing Schr\"{o}dinger's prior model $\bm{f}=\bm{0},\bm{g}=$ identity. 
In such generality, one needs to replace the heat kernel $\kappa_0$ with the corresponding $(\bm{f},\bm{g})$-dependent \emph{Markov kernels} associated with the forward and backward Kolmogorov PDEs, which are \emph{advection-diffusion PDEs}.

In this work, we investigate a different type of generalization of the classical SB \eqref{SBclassical}: we add a quadratic regularization term $q(t,\bm{x}_t^{\bm{u}}):=\frac{1}{2}\left(\bm{x}_t^{\bm{u}}\right)^{\top}\bm{Q}\bm{x}_t^{\bm{u}}$ with fixed $\bm{Q}\succeq\bm{0}$, in the integrand in \eqref{SBclassicalObj}. In other words, we modify the objective \eqref{SBclassicalObj} to $\int_{\mathbb{R}^{n}}\int_{t_0}^{t_1}\left(\frac{1}{2}\vert\bm{u}\vert^2+\frac{1}{2}\left(\bm{x}_t^{\bm{u}}\right)^{\top}\bm{Q}\bm{x}_t^{\bm{u}}\right)\:{\rm{d}}t\:\differential\mu_t$ but keep the constraints \eqref{SBclassicalPDEConstr}-\eqref{SBclassicalEndpointConstr} unchanged. We propose such regularization as a mechanism to bias the corrected process $\bm{x}_t^{\bm{u}}$ toward some desired or nominal constant (w.l.o.g. zero) vector for all $t\in[t_0,t_1]$ in addition to fulfilling the endpoint statistics constraints; see Fig. \ref{Fig:1dSBPsamplepathsWithAndWithoutq}. Details on simulating the SB for $Q>0$ case as in Fig. \ref{Fig:1dSBPsamplepathsWithAndWithoutq} will be discussed toward the end of Sec. \ref{sec:Discussions}. 

From a control-theoretic viewpoint, the proposed quadratic regularization is natural since penalizing state excursions away from nominal can be a soft way to promote safety, for example. Our formulation is in similar spirit as classical linear quadratic optimal control \cite{anderson2007optimal} where the stage cost indeed comprises of two summands: a positive semi-definite state cost-to-go (here $\frac{1}{2}\left(\bm{x}_t^{\bm{u}}\right)^{\top}\bm{Q}\bm{x}_t^{\bm{u}}$) and a positive definite control cost-to-go (here $\frac{1}{2}\vert \bm{u}\vert^2$). The matrix $\bm{Q}\succeq 0$ expresses the relative modeling importance of one summand over another.

    
\noindent\textbf{Contributions.}
We show that as in classical SB, $\rho_t^{\rm{opt}}$ in our case still factorizes into the product of a $\kappa$-harmonic and a $\kappa$-coharmonic where $\kappa$ is the Markov kernel associated with a \emph{reaction-diffusion PDE} where the quadratic regularization term plays the role of (state dependent) reaction rate. We derive this new kernel $\kappa$ in closed form in terms of exponential and hyperbolic functions. We explain how this result is a nontrivial generalization of the heat kernel $\kappa_{0}$. We deduce properties of this new kernel $\kappa$, and point out how special cases of our new kernel recover certain known kernels in quantum mechanics. We explain how our results enable numerically solving the generalized SB via dynamic Sinkhorn recursion.

\noindent\textbf{Related works.} Beyond its relevance for learning and control in general as outlined earlier, one specific area where SB is finding recent applications is in diffusion models for generative AI \cite{de2021diffusion,chen2021likelihood,zhou2023denoising,peluchetti2023diffusion,shi2024diffusion,gushchin2024entropic,gushchin2024building,lee2024disco}. This is because \eqref{SBclassical} allows for generating samples from data distribution in not just finite, but in a \emph{given} finite time. This circumvents the need to run a forward-time SDE for sufficiently long and then to reverse that SDE over the same time horizon \cite{song2020score}.  

Classical SB has been generalized in multiple directions, e.g., with more general drifts and diffusions \cite{liu2022deep,caluya2021wasserstein,nodozi2022schrodinger,liu2023i2sb,teter2023contraction,nodozi2023neural,nodozi2023physics,vargas2024transport} and with additional bounded domain constraints \cite{caluya2021reflected,deng2024reflected}. However, works on SB with state cost-to-go are limited \cite{wakolbinger1990schrodinger,dawson1990schrodinger}, {\black{\cite[Ch. 2.6, 2.7]{nagasawa2012schrodinger}}}. In particular, investigating SB with quadratic state cost-to-go and deducing closed form kernels for the same, as done here, are novel.

{\black{We emphasize here that by finding the Markov kernel $\kappa$ in closed form, our results enable the numerical solution of SB with quadratic state cost for \emph{generic endpoint distributions $\mu_0,\mu_1$ with finite second moments}. The solution of SB with quadratic state cost in the special case of Gaussian $\mu_0,\mu_1$ has appeared before in \cite{chen2015optimalACC,8249875}. Different from these prior works, the closed form Markov kernel derived here makes no parametric assumptions on $\mu_0,\mu_1$.}}

\noindent\textbf{Organization.} In Sec. \ref{sec:NotationsAndPrelim}, we set up basic notations and background ideas that will be used in the following development. Sec. \ref{sec:SBwithStateCost} spells out the SB problem formulation with generic state cost-to-go, explains the solution structure and dynamic Sinkhorn recursions to compute the solution as well the meaning of ``exactly solvable''. {\black{Readers already familiar with the SB ideas, could skip to Sec. \ref{subsec:ExactlySolvable} where new developments commence.}} Our main results appear in Sec. \ref{sec:MainResults}. Proofs for all the statements in Sec. \ref{sec:SBwithStateCost} and Sec. \ref{sec:MainResults} are deferred to Appendix \ref{AppSec:Proofs}. In Sec. \ref{sec:Discussions}, we discuss the derived results, their properties, and provide numerical results to illustrate the same. We conclude with some future directions.


\section{Notations and preliminaries}\label{sec:NotationsAndPrelim}
We use $\complexi$, $\overline{\mathbb{R}}$, $\vert\cdot\vert$, $\langle\cdot,\cdot\rangle$, {\black{$\mathcal{T}$}}, $\oslash$, $\det$, $\bm{0}$, $\bm{I}$ to denote the complex number $\sqrt{-1}$, the set of extended reals, Euclidean magnitude, Euclidean inner product, {\black{tangent bundle}}, elementwise (Hadamard) division, determinant, zero matrix or vector of appropriate size, identity matrix of suitable size, respectively. 
The symbols $\nabla$ and $\Delta$ denote the Euclidean gradient and Laplacian operators, respectively. Sometimes we put subscripts to these operators to disambiguate w.r.t. which vectors these differential operators are taken. The symbols $\exp$ and $e$ are used interchangeably to denote the exponential. The subscripted index vector notation $\bm{x}_{\left[i_{1}:i_{r}\right]}$ stands for the appropriate length $r$ sub-vector of the vector $\bm{x}$. Let $\mathbb{N}:=\{1,2,\hdots\}$, $\mathbb{N}_{0}:=\mathbb{N}\cup\{0\}$, and the $n$-fold tensor product $\mathbb{N}_{0}^{n}:=\underbrace{\mathbb{N}_{0}\times\mathbb{N}_{0}\times\hdots\mathbb{N}_{0}}_{n\;\text{times}}$.

We use inequalities $\succ$ and $\succeq$ in the sense of L\"{o}wner partial order. We will need hyperbolic sine, cosine, tangent, cotangent, secant and co-secant functions, denoted respectively as $\sinh(\cdot)$, $\cosh(\cdot)$, $\tanh(\cdot)$, $\coth(\cdot)$, ${\rm{sech}}(\cdot)$, and ${\rm{csch}}(\cdot)$. For matricial arguments, these functions are understood elementwise. The symbol $\bm{1}_{\{\cdot\}}$ is used to denote an indicator function for the condition or set in the subscript. We use ${\rm{eig}}\left(\cdot\right)$ to denote the eigenvalue operator. The normal PDF with mean $\bm{\mu}\in\mathbb{R}^{n}$ and covariance $\bm{\Sigma}\succ\bm{0}$, is denoted as $\mathcal{N}(\bm{\mu},\bm{\Sigma}):=(2\pi)^{-n/2}\det(\bm{\Sigma})^{-1/2}\exp\left(-\frac{1}{2}(\bm{x}-\bm{\mu})^{\top}\bm{\Sigma}^{-1}(\bm{x}-\bm{\mu})\right)$.

The (scaled) \emph{heat kernel} in $\mathbb{R}^{n}$ associated with the diffusion process $\differential\bm{x}_t = \sqrt{2}\differential\bm{w}_t$ is denoted as $\kappa_0$, i.e.,
\begin{align}
\kappa_0(s, \bm{x}, t, \boldsymbol{y}) = (4\pi(t-s))^{-n/2}\exp\!\left(\!-\dfrac{\vert \bm{x}-\bm{y}\vert^2}{4(t-s)}\!\right),\; 0\leq s<t<\infty,\;\forall\bm{x},\bm{y}\in\mathbb{R}^{n}.
\label{defHeatKernel}
\end{align}
Given two probability measures $\mathbb{P},\mathbb{Q}$ on some measure space, the \emph{relative entropy} a.k.a. \emph{Kullback-Leibler divergence} $D_{\rm{KL}}\left(\mathbb{P}\parallel\mathbb{Q}\right) := \mathbb{E}_{\mathbb{P}}\left[\log\dfrac{\differential\mathbb{P}}{\differential\mathbb{Q}}\right]$ where $\dfrac{\differential\mathbb{P}}{\differential\mathbb{Q}}$ denotes the Radon-Nikodym derivative.

We use the \emph{physicist's Hermite polynomials} \cite[Ch. 22]{AbraSteg72} of degree $n\in\mathbb{N}_{0}$, given by
\begin{align}
    H_{n}(x) := (-1)^n e^{x^2} \dfrac{\differential^n}{\differential x^n} \left( e^{-x^2} \right),    \label{defPhysicistHermitPoly}
\end{align}
which are orthogonal in the sense
\begin{align}
    \int_{-\infty}^{\infty} H_{m}(x) H_n(x) e^{-x^2}\differential x = \begin{cases} 0 &\textrm{if} \quad m \neq n,\\ 
    \sqrt{\pi} 2^n n! &\textrm{otherwise.}
    \end{cases}
    \label{OrthogonalityHermitePoly}
\end{align}
The following Lemma will be useful in our development.
\begin{lemma}[\textbf{A series sum for exponential of negative quadratic}]\label{LemmaExponentialOfQuadAsSeries}
For $\alpha,\beta, t_0, x$ fixed, and $0\leq t_0 < t < \infty$,
\begin{align}
    \sum_{j=0}^{\infty} \frac{(\alpha\left(t-t_{0}\right))^j}{j!\beta^{j}} \dfrac{\differential^{j}}{\differential x^{j}}e^{-(\beta x)^2} = e^{-\left(\beta x+ \alpha\left(t-t_{0}\right)\right)^2}.
\label{TaylorExpansion} 
\end{align}
\end{lemma}
\begin{proof}
The result follows from Taylor expansion of $e^{-\left(\beta x+ \alpha\left(t-t_{0}\right)\right)^2}$ in variable $t$ about $t_{0}$, and then identifying the coefficients as suitable order derivatives w.r.t. $x$.\qedsymbol
\end{proof}

For any $n\in\mathbb{N}$, the set of \emph{symplectic matrices}
\begin{align}
{\rm{Sp}}\left(2n,\mathbb{R}\right) := \bigg\{\bm{S}\in\mathbb{R}^{2n\times2n}\mid \bm{S}^{\top}\bm{JS}=\bm{J}\;\text{where}\;\bm{J}:=\begin{pmatrix}
\bm{0} & \bm{I}\\
-\bm{I} & \bm{0}
\end{pmatrix}\bigg\}
    \label{defSymplecticGroup}
\end{align}
is a connected Lie group that is a subgroup of ${\rm{SL}}(2n,\mathbb{R})$, the set of volume and orientation preserving linear maps in $\mathbb{R}^{2n}$. We denote the set of $n\times n$ symmetric positive definite matrices as ${\rm{Sym}}_{++}(n)$, i.e., $\bm{A}\in{\rm{Sym}}_{++}(n)$ means $\bm{A}\succ\bm{0}$. Let ${\rm{Sp}}_{++}\left(2n,\mathbb{R}\right):={\rm{Sp}}\left(2n,\mathbb{R}\right)\cap {\rm{Sym}}_{++}(2n)$.

We will use the following result, which is sometimes referred to as the ``central identity'' of quantum field theory (QFT) \cite[Appendix A]{zee2010quantum}.
\begin{lemma}[\textbf{Central identity of QFT}]\label{Lemma:CentralIdentityQFT}\cite[p. 2]{zinn2021quantum}\cite[p. 15]{zee2010quantum}
For a suitably smooth function $f:\mathbb{R}^{n}\mapsto\mathbb{R}$, and matrix $\bm{A}\in{\rm{Sym}}_{++}(n)$, 
\begin{align}
\displaystyle\int_{\mathbb{R}^n}\!\!\exp\!\left(\!-\frac{1}{2}\bm{x}^{\top}\bm{A}\bm{x}\!\right)\! f(\bm{x})\differential\bm{x}=\sqrt{ \frac{(2\pi)^{n}}{\det(\bm{A})}} \exp\left(\frac{1}{2}\nabla_{\bm{x}}^{\top}\bm{A}^{-1}\nabla_{\bm{x}}\right) f(\bm{x})\bigg\vert_{\bm{x}=\bm{0}},
\label{CentralIdentityQFT}
\end{align}
{\black{where the exponential of a differential operator is understood as a power series.}} In particular, for $f(\bm{x})=\exp\langle \bm{b},\bm{x}\rangle$, $\bm{b}\in\mathbb{R}^{n}$, we have $\int_{\mathbb{R}^n}\!\exp\!\left(\!-\frac{1}{2}\bm{x}^{\top}\bm{A}\bm{x}+\langle \bm{b},\bm{x}\rangle\!\right)\! \differential\bm{x}=\sqrt{ \frac{(2\pi)^{n}}{\det(\bm{A})}} \exp\left(\frac{1}{2} \bm{b}^{\top}\bm{A}^{-1} \bm{b}\right)$.
\end{lemma}

For diagonalizable $\bm{X}\in\mathbb{R}^{n\times n}$ with $\bm{X}=\bm{VDV}^{-1}$, and function $g(\cdot)$ well-defined on the spectrum of $\bm{X}$, the matrix function 
\begin{align}
g\left(\bm{X}\right)=\bm{V}g(\bm{D})\bm{V}^{-1},
\label{MatrixFunction}    
\end{align}
see e.g., \cite[Ch. 5]{gantmakher2000theory}, \cite[Ch. 1]{higham2008functions}.


\section{Schr\"{o}dinger bridge with state cost-to-go}\label{sec:SBwithStateCost}
We start with the problem formulation and its solution structure for generic state cost-to-go $q:[t_0,t_1]\times\mathbb{R}^{n}\mapsto\overline{\mathbb{R}}$ before delving into the specifics for the quadratic case (Sec. \ref{sec:MainResults}). 


\subsection{Schr\"{o}dinger factors and reaction-diffusion PDEs}
Hereafter, we assume that the endpoint distributions $\mu_0,\mu_1$ are absolutely continuous with respective PDFs $\rho_0,\rho_1$, and generalize the classical SB \eqref{SBclassical} with the addition of state cost-to-go $q:[t_0,t_1]\times\mathbb{R}^{n}\mapsto\overline{\mathbb{R}}$, as
\begin{subequations}
\begin{align}
&\underset{\left(\rho_t,\bm{u}\right)}{\arg\inf}\quad\displaystyle\int_{\mathbb{R}^{n}}\int_{t_0}^{t_1}\left(\dfrac{1}{2}\vert \bm{u}\vert^{2} + q\left(t,\bm{x}_t^{\bm{u}}\right)\right)\differential t\:\rho_t(t,\bm{x}_t^{u})\differential\bm{x}_t^{u} \label{LambertSBPobj_General}\\  
\text{subject to}\quad &\dfrac{\partial\rho_t}{\partial t} = - \nabla\cdot \left(\rho_t\bm{u}\right) + \Delta\rho_t, \label{LambertSBPdynconstr}\\
& \rho_t(t=t_0,\cdot) = \rho_0\:\text{(given)}, \quad \rho_t(t=t_1,\cdot) = \rho_{1}\:\text{(given)}\label{LambertSBPterminalconstr}.
\end{align}
\label{LambertSBP}
\end{subequations}
We first note that, with mild assumptions, this problem is well-posed (proof in Appendix \ref{Proof_of_ThmLSBPExistenceUniqueness}). 

\begin{theorem}[\textbf{Existence-Uniqueness of solution for SB with state cost-to-go}]\label{ThmLSBPExistenceUniqueness}
If $\rho_0,\rho_1$ have finite second raw moments and $q$ is measurable, then the minimizing tuple $\left(\rho^{\rm{opt}}_t,\bm{u}^{\rm{opt}}\right)$ for problem \eqref{LambertSBP} exists and is unique.
\end{theorem}
The unique solution to \eqref{LambertSBP} can be determined in terms of the harmonic and coharmonic factors $\varphi,\widehat{\varphi}$, a.k.a. \emph{Schr\"{o}dinger factors}, as follows (proof in Appendix \ref{Proof_of_ThmHopfColeReactionDiffusion}). {\black{Variants of this result have appeared before; see e.g., \cite{chen2015optimalACC}, \cite[Sec. 5]{chen2021stochastic}.}}

\begin{theorem}[\textbf{Minimizer of SB problem with state cost-to-go}]\label{ThmHopfColeReactionDiffusion}
 Consider the pair of Schr\"{o}dinger factors $(\widehat{\varphi},\varphi)$ that solve the system of forward and backward linear reaction-diffusion PDEs:
\begin{subequations}
\begin{align}
\displaystyle\frac{\partial\widehat{\varphi}}{\partial t} &= \underbrace{\left(\Delta - q\right)}_{{\red{\mathcal{L}_{\rm{forward}}}}}\widehat{\varphi},    
\label{FactorPDEForward}\\
\displaystyle\frac{\partial\varphi}{\partial t} &= \underbrace{\left(-\Delta + q\right)}_{{\red{\mathcal{L}_{\rm{backward}}}}}\varphi,
\label{FactorPDEBackward}
\end{align}
\label{FactorPDEs}
\end{subequations}
with coupled boundary conditions
\begin{align}
\widehat{\varphi}(t=t_0,\cdot)\varphi(t=t_0,\cdot) = \rho_0, \quad \widehat{\varphi}(t=t_1,\cdot)\varphi(t=t_1,\cdot) = \rho_1.
\label{factorBC}    
\end{align}
For all $t\in[t_0,t_1]$, the minimizing pair for \eqref{LambertSBP} is
\begin{subequations}
\begin{align}
\rho^{\rm{opt}}_t(t,\cdot) &= \widehat{\varphi}(t,\cdot)\varphi(t,\cdot) \label{OptimallyControlledJointPDF},\\
\bm{u}^{\rm{opt}}(t,\cdot) &= \nabla_{(\cdot)}\log\varphi(t,\cdot).\label{OptimalVelocity}
\end{align}
\label{RecoverOptimalSolution}
\end{subequations}
\label{linear_PDEs}
\end{theorem}

\begin{remark}\label{ReamrkBornRelation}
The PDEs \eqref{FactorPDEs} have striking similarities with the (time dependent) Schr\"{o}dinger equations in quantum mechanics:
\begin{align}
-\complexi\displaystyle\frac{\partial\widehat{\Psi}}{\partial \tau} = \left(\Delta - q\right)\widehat{\Psi}, \qquad   
\complexi\displaystyle\frac{\partial\Psi}{\partial \tau} = \left(-\Delta + q\right)\Psi, \qquad \complexi:=\sqrt{-1}.
\label{QuantumFactorPDEs}
\end{align}
The PDEs \eqref{QuantumFactorPDEs} can be obtained from \eqref{FactorPDEs} under imaginary time change (Wick's rotation) $t\mapsto \complexi\tau$, where $q(\cdot)$ plays the role of quantum mechanical potential. Then, $\Psi,\widehat{\Psi}$ are the wave function and its adjoint, and their product $\widehat{\Psi}(t,\cdot)\Psi(t,\cdot)$ returns the probability amplitude as in \eqref{OptimallyControlledJointPDF}. In other words, the Schr\"{o}dinger factors $(\widehat{\varphi},\varphi)$ in Theorem \ref{ThmHopfColeReactionDiffusion} are classical analogues of the wave functions, and the relation \eqref{OptimallyControlledJointPDF} is the classical analogue of the celebrated Born's relation \cite{born1926quantenmechanik}. {\black{For further discussions relating \eqref{FactorPDEs} and \eqref{QuantumFactorPDEs}, and their solutions, see \cite[Ch. 1.4, 2.1]{aebi1996schrodinger} .}}
\end{remark}

\begin{remark}\label{ReamrkCreationKilling}
Unlike the classical SB, the PDEs \eqref{FactorPDEs} do not form an $L_2$-adjoint pair due to the change in sign in reaction terms among the two. From a probabilistic viewpoint, the $\pm q\widehat{\varphi}$ terms represent state-time-dependent creation or killing of probability mass with rate $q$.
\end{remark}

\subsection{Dynamic Sinkhorn recursion}

\begin{figure}[htbp]
  \centering
  \includegraphics[width=0.7\textwidth]{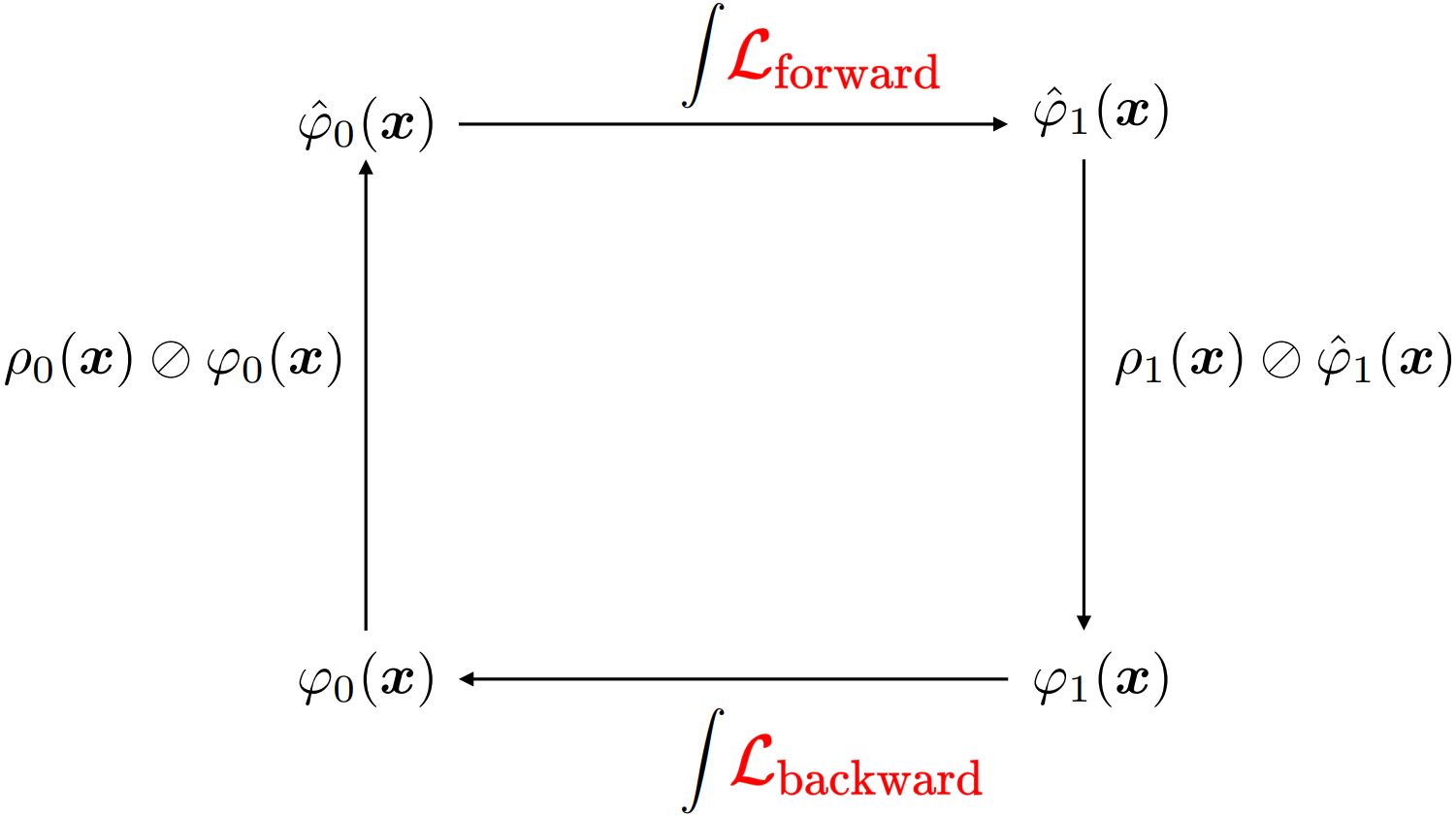}
  \caption{Dynamic Sinkhorn recursion to solve SBP with linear convergence in Hilbert metric. In this work, the $\mathcal{L}_{\rm{forward}}, \mathcal{L}_{\rm{backward}}$ are as in \eqref{FactorPDEs}.}
  \label{fig:FixedPointRecursion}
\end{figure}

Even the classical SB ($q=0$) has no analytical solution for generic problem data: $\rho_{0},\rho_{1}$ with finite second moments. Standard approach to numerically solve SB problems is to use dynamic Sinkhorn recursions which involve forward and backward PDE initial value problems (IVPs) associated with \eqref{FactorPDEs}. To see this, let $\widehat{\varphi}_0(\cdot) := \widehat{\varphi}(t=t_0,\cdot)$ and $\varphi_1(\cdot):=\varphi(t=t_1,\cdot)$. If $\kappa$ is the Markov kernel or Green's function associated with \eqref{FactorPDEs}, then solving the SB problem can be reduced to the computation of the function pair $\left(\widehat{\varphi}_{0}(\cdot),\varphi_{1}(\cdot)\right)$ from the following system of coupled nonlinear integral equations:
\begin{subequations}
\begin{align}
& \rho_0(\bm{x})=\widehat{\varphi}_{0}(\bm{x}) \int_{\mathbb{R}^n} \kappa(t_0, \bm{x}, t_1, \boldsymbol{y}) \varphi_{1}(\boldsymbol{y}) \mathrm{d} \boldsymbol{y}, \label{backward}\\
& \rho_1(\bm{x})=\varphi_{1}(\bm{x}) \int_{\mathbb{R}^n} \kappa(t_0, \boldsymbol{y}, t_1, \bm{x}) \widehat{\varphi}_{0}(\boldsymbol{y}) \mathrm{d} \boldsymbol{y},\label{forward}
\end{align}
\label{SchrodingerSystemGeneral}
\end{subequations}
known as the \emph{Schr\"{o}dinger system}. This system \eqref{SchrodingerSystemGeneral} can in turn be solved via a fixed point recursion over $\left(\widehat{\varphi}_{0}(\cdot),\varphi_{1}(\cdot)\right)$ shown in Fig. \ref{fig:FixedPointRecursion} that is known \cite{chen2016entropic} to be contractive in Hilbert's projective metric \cite{bushell1973hilbert}.

Specifically, one makes everywhere positive initial guess of $\widehat{\varphi}_0(\cdot)$, then integrates $\widehat{\varphi}(t,\cdot)$ forward from $t_0$ to $t_1$ to determine $\widehat{\varphi}_1(\cdot):=\widehat{\varphi}(t_1,\cdot)$. Since $\rho_1(\cdot)$ is known, enforcing boundary condition \eqref{factorBC} at $t=t_1$ gives $\varphi_1(\cdot) = \rho_1(\cdot) \oslash \widehat{\varphi}_1(\cdot)$. With this initial value, integrating $\varphi(t,\cdot)$ backward in time yields $\varphi_0(\cdot):=\varphi(t_0,\cdot)$. With known $\rho_0(\cdot)$, we return to $\widehat{\varphi}_0(\cdot)$ by evaluating $\rho_0(\cdot) \oslash \varphi_0(\cdot)$, concluding one epoch of the fixed point recursion in Fig. \ref{fig:FixedPointRecursion}. The converged $\left(\widehat{\varphi}_{0}(\cdot),\varphi_{1}(\cdot)\right)$ thus determined, are then used to compute the Schr\"{o}dinger factors at arbitrary $t \in [t_0, t_1]$ as
\begin{subequations}
\begin{align}
& \widehat{\varphi}(t,\bm{x}) := \int_{\mathbb{R}^n} \kappa(t_0, \boldsymbol{y}, t, \bm{x}) \widehat{\varphi}_0(\boldsymbol{y}) \mathrm{d} \boldsymbol{y}, \quad t \geq t_0, \label{SchrodingerFactorsForward}\\
& \varphi(t,\bm{x}) := \int_{\mathbb{R}^n} \kappa(t, \boldsymbol{x}, t_1, \boldsymbol{y}) \varphi_1(\boldsymbol{y}) \mathrm{d} \boldsymbol{y}, \quad t \leq t_1. \label{SchrodingerFactorsBackward}
\end{align}
\label{SchrodingerFactors}
\end{subequations}
The forward and backward integration needed in the fixed point recursion described above are special cases of \eqref{SchrodingerFactorsForward} and \eqref{SchrodingerFactorsBackward}, respectively. Therefore, determining the appropriate Markov kernel $\kappa$ facilitates the solution to the SB with state cost.

\subsection{Exactly solvable SBs and closed form kernels}\label{subsec:ExactlySolvable} We say that an SB problem is \emph{exactly solvable} if the corresponding (uncontrolled) Markov kernel $\kappa$ is known in closed form\footnote{In view of Remark \ref{ReamrkBornRelation}, our nomenclature matches that in quantum mechanics. For example, the PDEs \eqref{QuantumFactorPDEs} are said to be \emph{exactly solvable} if the corresponding Green's functions or ``propagators'' \cite[Ch. 2.6]{sakurai2020modern} can be determined in closed form.} \cite{borwein2013closed}, i.e., without integrals or series sum. This is the case for classical SB where $\kappa \equiv \kappa_{0}$, the heat kernel in \eqref{defHeatKernel}. 

This is also the case for \emph{linear SB}, i.e., when the prior physics is of the form $\differential\bm{x}_t=\bm{A}(t)\bm{x}_t \differential t + \sqrt{2}\bm{B}(t)\differential\bm{w}_t$ where the matricial trajectory pair $(\bm{A}(t),\bm{B}(t))$ is continuous and bounded for all $t\in[t_0,t_1]$, the associated state transition matrix is $\bm{\Phi}_{t\tau}:=\bm{\Phi}(t,\tau)$ $\forall t_0\leq \tau\leq t\leq t_1$, and the pair $(\bm{A}(t),\bm{B}(t))$ controllable in the sense that the controllability Gramian
$$\bm{\Gamma}_{t_{1}t_{0}}:=\!\int_{t_0}^{t_1}\!\!\bm{\Phi}_{t_{1}\tau}\bm{B}(\tau)\bm{B}^{\top}(\tau)\bm{\Phi}_{t_{1}\tau}^{\top}\differential\tau \succ \bm{0}.$$
In this case, the Markov kernel $\kappa \equiv \kappa_{\rm{Linear}}$ is \cite[Sec. III]{teter2023contraction}
\begin{align}
\kappa_{\rm{Linear}}\left(t_0,\bm{x}_0,t_1,\bm{x}_1\right) =& \left(4\pi(t_1 - t_0)\right)^{-n/2}\det\left(\bm{\Gamma}_{t_{1}t_{0}}\right)^{-1/2}\nonumber\\
&\exp\left(-\dfrac{\left(\bm{\Phi}_{t_{1}t_{0}}\bm{x}_0 - \bm{x}_{1}\right)^{\top}\bm{\Gamma}_{t_{1}t_{0}}^{-1}\left(\bm{\Phi}_{t_{1}t_{0}}\bm{x}_0 - \bm{x}_{1}\right)}{4(t_1-t_0)}\right).
\label{LTVkernel}    
\end{align}

A closed form $\kappa$ such as \eqref{defHeatKernel} or \eqref{LTVkernel} helps {\black{numerically}} implement the forward and backward pass in dynamic Sinkhorn recursion (Fig. \ref{fig:FixedPointRecursion}) via simple matrix-vector multiplications. This obviates {\black{randomized computation with}} Feynman-Kac path integrals or other nontrivial solvers \cite{caluya2021wasserstein,teter2024solution} which incur function approximation errors.

To the best of our knowledge, no other \emph{exactly solvable} model for SB is known. We next derive $\kappa$ in closed form for SB with quadratic state cost-to-go, i.e., for \eqref{FactorPDEs} with $q(t,\bm{x}) = \frac{1}{2}\bm{x}^{\top}\bm{Q}\bm{x}$, $\bm{Q}\succeq \bm{0}$. 

\section{Main Results}\label{sec:MainResults}
In this section, we first deduce the Markov kernel for the SB problem in closed form with quadratic state cost $q(t,\bm{x}) = \frac{1}{2}\bm{x}^{\top}\bm{Q}\bm{x}$ where $\bm{Q}\succ\bm{0}$ (Theorem \ref{Thm:Kernel}). We denote this kernel as $\kappa_{++}$. We then generalize our results for $\bm{Q}\succeq\bm{0}$ (Theorem \ref{SemiPosDefKernel}) to derive the corresponding kernel $\kappa_{+}$.

It suffices to derive the Markov kernel for the reaction-diffusion PDE   \eqref{FactorPDEForward} since we may apply the same result to determine a solution to the PDE \eqref{FactorPDEBackward} under the change-of-variable $t\mapsto\overline{t} := t_1 - (t - t_0)$. So w.l.o.g., we restrict our derivation of kernel to \eqref{FactorPDEForward}.

\subsection{Kernel for \texorpdfstring{$\bm{Q}\succ\bm{0}$}{Q}}
For $\bm{Q}\succ\bm{0}$, recall the eigen-decomposition $\frac{1}{2}\bm{Q} = \bm{V}^{\top}\bm{D}\bm{V}$ where the positive diagonal matrix $\bm{D}$ has the eigenvalues of $\frac{1}{2}\bm{Q}$ along its main diagonal, and the columns of the orthogonal matrix $\bm{V}$ are the eigenvectors of $\frac{1}{2}\bm{Q}$. For convenience, let  $d_{i}>0$ be the $i$th diagonal entry of $\bm{D}$, let $\bm{y} := \bm{V}\bm{x}$, and 
\begin{align}
\widehat{\eta}(t,\bm{y}):=\widehat{\varphi}\left(t,\bm{x}=\bm{V}^{\top}\bm{y}\right).
\label{defEtaHat}
\end{align}
The change-of-variable $\bm{x}\mapsto\bm{y}$ gives $\Delta_{\bm{x}} \widehat{\varphi}(t,\bm{x}) = \Delta_{\bm{y}} \widehat{\eta}(t,\bm{y})$, and \eqref{FactorPDEForward} becomes
\begin{align}
    \frac{\partial \widehat{\eta}}{\partial t} &=  \Delta_{\bm{y}} \widehat{\eta} - \left(\bm{y}^{\top}\bm{D} \bm{y}\right)\widehat{\eta}    \label{FactorPDEForwardNew}\\
    &=\displaystyle\sum_{i=1}^{n}\left(\dfrac{\partial^2}{\partial y_i^2} - d_i y_i^2\right)\widehat{\eta}.\nonumber
\end{align}
We seek a separation-of-variables solution ansatz $\widehat{\eta}(t, \bm{y}) = T(t) \prod_{i=1}^{n}Y_{i}(y_{i})$ for \eqref{FactorPDEForwardNew}, which requires us to satisfy 
$\frac{1}{T} \frac{\differential T}{\differential t} = \sum_{i = 1}^{n}\left( \frac{1}{Y_i} \frac{\differential^2 Y_i}{\differential y_i^2}  - d_i y_i^2 \right) = c$ for some constant $c\in\mathbb{R}$. Equivalently, we need to satisfy the $n+1$ ODEs
\begin{align*}
    &\dfrac{\differential T}{\differential t} = cT,\\
&\dfrac{\differential^2 Y_1}{\differential y_1^2}  - \left(d_1 y_1^2 + c_1\right)Y_1 = 0,\\
    &\quad\vdots \\
    &\dfrac{\differential^2 Y_n}{\differential y_n^2}  - \left(d_n y_n^2 + c_n\right)Y_n = 0,
    \end{align*}
for constants $c,c_1,\hdots,c_n\in\mathbb{R}$ subject to the constraint $c_1 + \ldots + c_n = c$. 

The time ODE has solution $T(t) = ke^{c(t-t_0)}$ for some constant $k \in \mathbb{R}$. To determine each $Y_k(y_k)$, we use the following Lemma (proof in Appendix \ref{proof_of_SolToSpatialODE}).
\begin{lemma}[\textbf{Solution of a parametric second order nonlinear ODE}] For parameters $c\in\mathbb{R}$ and $d>0$, the solutions to the parametric ODE
\begin{align}
    \dfrac{\differential^2 Y}{\differential y^2}  - (d y^2 + c)Y = 0,
    \label{SpatialODE}
\end{align}
are given by
\begin{align}
    Y = a \exp{\left(- \frac{\sqrt{d} y^2}{2}\right)} H_{n}\left(d^{1/4} y \right),\quad\text{for any}\quad n = - \frac{c}{2\sqrt{d}} - \frac{1}{2} \in \mathbb{N}_{0},\quad a\in\mathbb{R}.
    \label{SolToSpatialODE}
\end{align}
\label{Lemma:SolToSpatialODE}
\end{lemma}


Applying \eqref{SolToSpatialODE} to the ODEs for $Y_1, Y_2,\hdots, Y_n$, we find that the general solution to \eqref{FactorPDEForwardNew} is of the form
\begin{align}
    \widehat{\eta}(t, \bm{y} ) = \!\sum_{m_n = 0}^{\infty}\! \ldots \!\sum_{m_1 = 0}^{\infty}\! a_{\bm{m}} \prod_{i = 1}^{n} \exp{\left(\!- \left(t-t_{0}\right)\sqrt{d_i}\left(2 m_i + 1\right) - \frac{y_i^2 \sqrt{d_i}}{2}\right)} H_{m_i}\!\left( d_i^{1/4}y_i \right)\!,
    \label{general_solution}
\end{align}
where the index vector $\bm{m} := (m_1, m_2, \ldots, m_n) \in \mathbb{N}_{0}^{n}$, and the coefficients $a_{\bm{m}}$ remain to be determined.
Using \eqref{general_solution} together with the orthogonality of Hermite polynomials, we express the IVP solution to \eqref{FactorPDEForwardNew} with known initial condition $\widehat{\eta}_{0}(\bm{y}):=\widehat{\eta}(t_{0},\bm{y}) =\widehat{\varphi}_{0}\left(\bm{V}^{\top}\bm{y}\right)$ as follows (proof in Appendix \ref{proof_of_ParticularSol}).
\begin{proposition}
[\textbf{IVP solution for \eqref{FactorPDEForwardNew} with positive diagonal $\bm{D}$ as series sum}]
For $0\leq t_{0} < t < \infty$ and initial condition $\widehat{\eta}_{0}(\bm{y}):=\widehat{\eta}(t_{0},\bm{y})$, the unique solution to \eqref{FactorPDEForwardNew} with positive diagonal matrix $\bm{D}$ is given by
\begin{align}
    \widehat{\eta}(t, \bm{y} ) = \sum_{m_n = 0}^{\infty} \ldots \sum_{m_1 = 0}^{\infty} \int_{-\infty}^{\infty} \ldots \int_{-\infty}^{\infty} (*) \:\widehat{\eta}_{0}(\bm{z})\: \differential z_1 \ldots \differential z_n     \label{general_solution_with_coeff}
\end{align}
where
\begin{align}
    (*) :=\prod_{i = 1}^{n} \frac{d_i^{1/4}} {\sqrt{\pi} 2^{m_i}\left(m_{i}!\right)} \exp{\left(- \left(t-t_{0}\right)\sqrt{d_i}(2m_i + 1) - \frac{(y_i^2+z_i^2) \sqrt{d_i}}{2}  \right)}\label{integrand}\\
\hspace{200pt}H_{m_i}\hspace{-1pt}\left(\!d_i^{1/4}y_i\!\right)  \hspace{-2pt}H_{m_i}\hspace{-3pt}\left(\!d_i^{1/4}z_i\!\right).\nonumber
\end{align}
\label{ParticularSol}
\end{proposition}
Our next Theorem (proof in Appendix \ref{proof_of_Thm:Kernel}) resolves the series sum in \eqref{general_solution_with_coeff}-\eqref{integrand} to derive the \emph{closed form} Markov kernel $\kappa\equiv\kappa_{++}$ for \eqref{FactorPDEForwardNew} as desired.
\begin{theorem}[\textbf{Closed form kernel $\kappa_{++}$ for IVP solution of \eqref{FactorPDEForwardNew} with positive diagonal $\bm{D}$}]\label{Thm:Kernel}
Consider the same setting as in Proposition \ref{ParticularSol}. The solution \eqref{general_solution_with_coeff} simplifies to
\begin{align}
    \widehat{\eta}(t, \bm{y} ) = \int_{-\infty}^{\infty} \ldots \int_{-\infty}^{\infty} \kappa_{++}(t_{0}, \bm{y}, t, \bm{z}) \:\widehat{\eta}_0(\bm{z})\: \differential z_1 \ldots \differential z_n
\label{EtaHatWithPosDefKernel}    
\end{align}
where the kernel
\begin{align}
    \kappa_{++}(t_{0}, \bm{y}, t, \bm{z}) = \frac{(\det({\black{\bm{M}_{tt_{0}}}}))^{1/4}}{(2\pi)^{n/2}\sqrt{\det(\sinh(2\left(t-t_{0}\right)\sqrt{\bm{D}}))}} \exp{\!\left(\!-\frac{1}{2} \begin{pmatrix} \bm{y}^{\top} & \bm{z}^{\top} \end{pmatrix} {\black{\bm{M}_{tt_{0}}}}\begin{pmatrix} \bm{y}\\ \bm{z} \end{pmatrix}\!\right)},
    \label{kernel_function}
\end{align}
and the matrix
\begin{align}
    {\black{\bm{M}_{tt_{0}}}} = \begin{bmatrix} \sqrt{\bm{D}} \coth{\left(2\left(t-t_{0}\right) \sqrt{\bm{D}}\right)} & - \sqrt{\bm{D}}  \: {\rm{csch}}\left(2\left(t-t_{0}\right) \sqrt{\bm{D}}\right) \\ -\sqrt{\bm{D}} \: {\rm{csch}}\left(2\left(t-t_{0}\right) \sqrt{\bm{D}}\right) & \sqrt{\bm{D}} \coth{\left(2\left(t-t_{0}\right) \sqrt{\bm{D}}\right)} \end{bmatrix}.  \label{Matrix_M_Def}
\end{align}
\end{theorem}
\begin{remark}\label{ReamrkNewKernelIsNotRadial}
We note that unlike the heat kernel \eqref{defHeatKernel}, $\kappa_{++}$ is not radial in the sense \eqref{kernel_function} is not a function of $\vert\bm{y}-\bm{z}\vert$ alone.
\end{remark}
The next Proposition (proof in Appendix \ref{ProofMatrixM}) states that the matrix ${\black{\bm{M}_{tt_{0}}}}$ in \eqref{Matrix_M_Def} can be seen as positive diagonal scaling of a symplectic positive definite matrix, and is therefore positive definite. This will be useful in establishing properties of $\widehat{\varphi}(t,\bm{x})$ in Proposition \ref{Prop:Properties}.
\begin{proposition}[\textbf{Properties of matrix ${\black{\bm{M}_{tt_{0}}}}$}]\label{Proposition:MatrixM}
For $0\leq t_0 < t < \infty$, and given positive diagonal matrix $\bm{D}$, the matrix ${\black{\bm{M}_{tt_{0}}}}$ in \eqref{Matrix_M_Def} satisfies\\
\noindent(i) decomposition: ${\black{\bm{M}_{tt_{0}}}} = \begin{bmatrix}
\bm{D}^{1/4} & \bm{0}\\
\bm{0} & \bm{D}^{1/4}
\end{bmatrix}{\black{\bm{M}^{(1)}_{tt_{0}}\bm{M}^{(2)}_{tt_{0}}}}\begin{bmatrix}
\bm{D}^{1/4} & \bm{0}\\
\bm{0} & \bm{D}^{1/4}
\end{bmatrix}$ where
\begin{align*}
{\black{\bm{M}^{(1)}_{tt_{0}}}}&=\begin{bmatrix}\cosh(2(t-t_0)\sqrt{\bm{D}})  & -\bm{I} \\  -\bm{I} &  \cosh(2(t-t_0)\sqrt{\bm{D}})  \end{bmatrix},\\
{\black{\bm{M}^{(2)}_{tt_{0}}}}&=\begin{bmatrix} {\rm{csch}}(2(t-t_0)\sqrt{\bm{D}})  & \bm{0} \\  \bm{0} & {\rm{csch}}(2(t-t_0)\sqrt{\bm{D}})\end{bmatrix},
\end{align*}
\noindent (ii) ${\black{\bm{M}^{(1)}_{tt_{0}}\bm{M}^{(2)}_{tt_{0}}}} \in {\rm{Sp}}_{++}\left(2n,\mathbb{R}\right)$,\\
\noindent (iii) ${\black{\bm{M}_{tt_{0}}}}\in{\rm{Sym}}_{++}\left(2n,\mathbb{R}\right)$.
\end{proposition}
{\black{Since $\bm{M}_{tt_{0}}\succ\bm{0}$, the exponential term in \eqref{kernel_function} is of the form $\exp\left(-\frac{1}{2}\dist_{tt_{0}}^2(\bm{y},\bm{z})\right)$ where $\dist_{tt_{0}}(\bm{y},\bm{z})$ is a distance metric. Our next result (proof in Appendix \ref{proof_of_prop_geodesic}) provides geometric insight that $\dist^{2}_{tt_{0}}(\bm{y},\bm{z})$ is, in fact, the minimal value of the action integral induced by the Lagrangian $L:\mathcal{T}\mathbb{R}^{n}\mapsto\mathbb{R}$ given by\footnote{Recall that the diagonal matrix associated with the eigen-decomposition of $\bm{Q}$ is $2\bm{D}$.} 
\begin{align}
L(\bm{q},\bm{v}) := \dfrac{1}{2}\vert\bm{v}\vert^2 + \bm{q}^{\top}\left(2\bm{D}\right)\bm{q}, \quad \left(\bm{q},\bm{v}\right)\in \mathcal{T}\mathbb{R}^{n}.
\label{defLagrangianGeodesic}    
\end{align}
\begin{proposition}[\textbf{Induced distance in $\kappa_{++}$}]\label{prop:Geodesic}
Given $\bm{y},\bm{z}\in\mathbb{R}^{n}$, the minimum value of the action integral associated with the Lagrangian \eqref{defLagrangianGeodesic}: \begin{align}
\dist_{tt_{0}}^2(\bm{y},\bm{z}):=\underset{\gamma(\cdot)\in\big\{\mathcal{C}^{2}\left([t_0,t]\right)\mid \gamma\left(t_0\right)=\bm{y}, \gamma\left(t\right)=\bm{z}\big\}}{\inf}\displaystyle\int_{t_0}^{t}L\left(\gamma(\tau),\dot{\gamma}(\tau)\right)\differential\tau
\label{ActionIntegral}    
\end{align}
equals to $\begin{pmatrix}\bm{y}^{\top} & \bm{z}^{\top} \end{pmatrix} \bm{M}_{tt_{0}}\begin{pmatrix} \bm{y} \\ \bm{z} \end{pmatrix}$, where $\bm{M}_{tt_{0}}$ is given by \eqref{Matrix_M_Def}.
\end{proposition}

\begin{remark}\label{Remark:IntegralNotEqualToOne}
Unlike the Markov kernels \eqref{defHeatKernel} and \eqref{LTVkernel}, the kernel $\kappa_{++}$ in \eqref{kernel_function} is not a transition probability density because of the reaction term contributing to state-dependent creation or killing of mass. This can be seen by applying \eqref{CentralIdentityQFT} to evaluate $\int_{\mathbb{R}^{2n}}\kappa_{++}(t_{0}, \bm{y}, t, \bm{z}) \differential\bm{y}\differential\bm{z}= \det\left(\sqrt{\bm{D}}\sinh\left(2(t-t_{0})\sqrt{\bm{D}}\right)\right)^{-1/2}\neq 1$.
\end{remark}
}}

\subsection{Known kernels as special cases of $\kappa_{++}$}\label{subsec:specialcases} In the special case $\frac{1}{2}\bm{Q}$, and hence $\bm{D}$, equals $\bm{I}$, direct substitution reduces \eqref{kernel_function} to the \emph{multivariate Mehler kernel} \cite{mehler1866ueber}, \cite[Thm. 1]{robert2021coherent} 
\begin{align}
\kappa_{\rm{Mehler}}(t_{0}, \bm{y}, t, \bm{z}) &= \dfrac{1}{\left(2\pi\sinh(2(t-t_0))\right)^{n/2}}\exp\left(\vphantom{\dfrac{1}{2}}{\rm{csch}}(2(t-t_0))\langle\bm{y},\bm{z}\rangle \right.\label{MultivariateMehlerKernel}\\
&\qquad\qquad\qquad\qquad\left.- \frac{1}{2}\coth\left(2(t-t_0)\right)\left(\vert \bm{y}\vert^2 + \vert\bm{z}\vert^2\right)\right),
\nonumber    
\end{align}
which is the propagator or Green's function for the \emph{isotropic quantum harmonic oscillator}. Well-known extensions of the Mehler kernel include the \emph{Kibble-Slepian formula} \cite{kibble1945extension,slepian1972symmetrized} and variants thereof \cite{louck1981extension,hormander1995symplectic,pravda2018generalized}.

Another special case of interest for $\kappa_{++}$ is the limit $\bm{D}\downarrow \bm{0}$ understood as $d_{i}\downarrow 0\:\forall i =1,\hdots,n$. Our next result (proof in Appendix \ref{proof_of_heat_kernel_spl_case}) is that in this limit, the kernel \eqref{kernel_function} recovers the heat kernel \eqref{defHeatKernel}. This will be useful in subsection \ref{subsec:PosSemiDef}.
\begin{theorem}[\textbf{Heat kernel as limiting case of} \texorpdfstring{$\kappa_{++}$}{k}]\label{ThmHeatKernelSplCase}
Consider the kernel $\kappa_{++}$ in \eqref{kernel_function}-\eqref{Matrix_M_Def} and the heat kernel $\kappa_0$ in \eqref{defHeatKernel}. Then, $\displaystyle\lim_{(d_1,\hdots,d_n)\downarrow\bm{0}}\kappa_{++}(t_0, \bm{y}, t, \bm{z}) =\kappa_{0}(t_0, \bm{y}, t, \boldsymbol{z})$.
\end{theorem}
{\black{That $\kappa_0$ arises as a limiting case of $\kappa_{++}$ is consistent with the principle of least action in Proposition \ref{prop:Geodesic} since in the limit $\bm{D}\downarrow\bm{0}$, the Lagrangian in \eqref{defLagrangianGeodesic} equals $\frac{1}{2}\vert\bm{v}\vert^2$. Then, \eqref{ActionIntegral} yields the squared Euclidean distance $\vert\bm{y}-\bm{z}\vert^2$ which is what appears inside the exponential in \eqref{defHeatKernel} up to the usual scaling.}}

\subsection{Kernel for \texorpdfstring{$\bm{Q}\succeq\bm{0}$}{Q}}\label{subsec:PosSemiDef}
In Theorem \ref{SemiPosDefKernel} next (proof in Appendix \ref{proof_of_SemiPosDefKernel}), we generalize the closed form Markov kernel for the case $\bm{Q}\succeq\bm{0}$. We denote this generalized kernel as $\kappa_{+}$. We re-index the spatial variables to account for that zero eigenvalues of $\frac{1}{2}\bm{Q}$ can be interlaced with its positive eigenvalues.

\begin{theorem}[\textbf{Closed form kernel $\kappa_{+}$ for IVP solution of \eqref{FactorPDEForwardNew} with nonnegative diagonal $\bm{D}$}]\label{SemiPosDefKernel}
Consider diagonal matrix $\bm{D}$ having $p\in\mathbb{N}_{0}$ zero diagonal entries where $p<n$, and $n-p$ positive diagonal entries. Let $i_{1}, i_{2}, \ldots, i_{n - p}$ be the indices corresponding to positive diagonal entries of $\bm{D}$. Likewise, let $i_{n - p + 1}, i_{n - p + 2},\allowbreak\ldots,i_{n}$ be the indices corresponding to zero diagonal entries of $\bm{D}$.

For $0\leq t_{0} < t < \infty$ and initial condition $\widehat{\eta}_{0}(\bm{y}):=\widehat{\eta}(t_{0},\bm{y})$, the unique solution to \eqref{FactorPDEForwardNew} with such diagonal matrix $\bm{D}$, is
\begin{align}
    \widehat{\eta}(t, \bm{y} ) = \int_{-\infty}^{\infty} \ldots \int_{-\infty}^{\infty} \kappa_{+}(t_{0}, \bm{y}, t, \bm{z}) \:\widehat{\eta}_0(\bm{z})\: \differential z_1 \ldots \differential z_n
\label{EtaHatWithSemiDefKernel}    
\end{align}
where the kernel \begin{align}\label{semi_definite_kernel}
        \kappa_{+}(t_0, \bm{y}, t, \bm{z}) = \kappa_{++}\left(t_0, \bm{y}_{[i_1:i_{n-p}]}, t, \bm{z}_{[i_1:i_{n-p}]}\right) \kappa_{0}\left(t_0, \bm{y}_{[i_{n-p+1}:i_n]}, t, \bm{z}_{[i_{n-p+1}:i_n]}\right),
    \end{align}
    and $\kappa_{++},\kappa_{0}$ are as in Theorem \ref{Thm:Kernel} and \eqref{defHeatKernel}, respectively.    
\end{theorem}
\begin{remark}\label{remark:FromEtaHatBackToPhiHat}
For a given $\bm{Q}\succeq\bm{0}$, the solution \eqref{EtaHatWithPosDefKernel} or \eqref{EtaHatWithSemiDefKernel}, can be brought back to the original $\bm{x}$ coordinates to determine the Schr\"{o}dinger factor
\begin{align}
\widehat{\varphi}(t,\bm{x}) = \widehat{\eta}(t,\bm{y}=\bm{Vx}) = \int_{-\infty}^{\infty} \ldots \int_{-\infty}^{\infty} \kappa_{+}(t_{0}, \bm{Vx}, t, \bm{z}) \:\widehat{\varphi}_0\left(\bm{V}^{\top}\bm{z}\right)\differential z_1 \ldots \differential z_n,
\label{PhiHatAsIntegral}
\end{align}
wherein $\kappa_{+}$ specializes to $\kappa_{++}$ when $\bm{Q}\succ\bm{0}$.
As mentioned earlier, the PDE \eqref{FactorPDEBackward} and its solution can be mapped to that of \eqref{FactorPDEForward} via transformation $t\mapsto t_0 + t_1 - t$; therefore, the other Schr\"{o}dinger factor $\varphi(t,\bm{x}) = \widehat{\eta}(t_0 + t_1 - t,\bm{y}=\bm{Vx})$. 
\end{remark}

\begin{figure}[t]
	\centering
        \includegraphics[
      width=\linewidth]{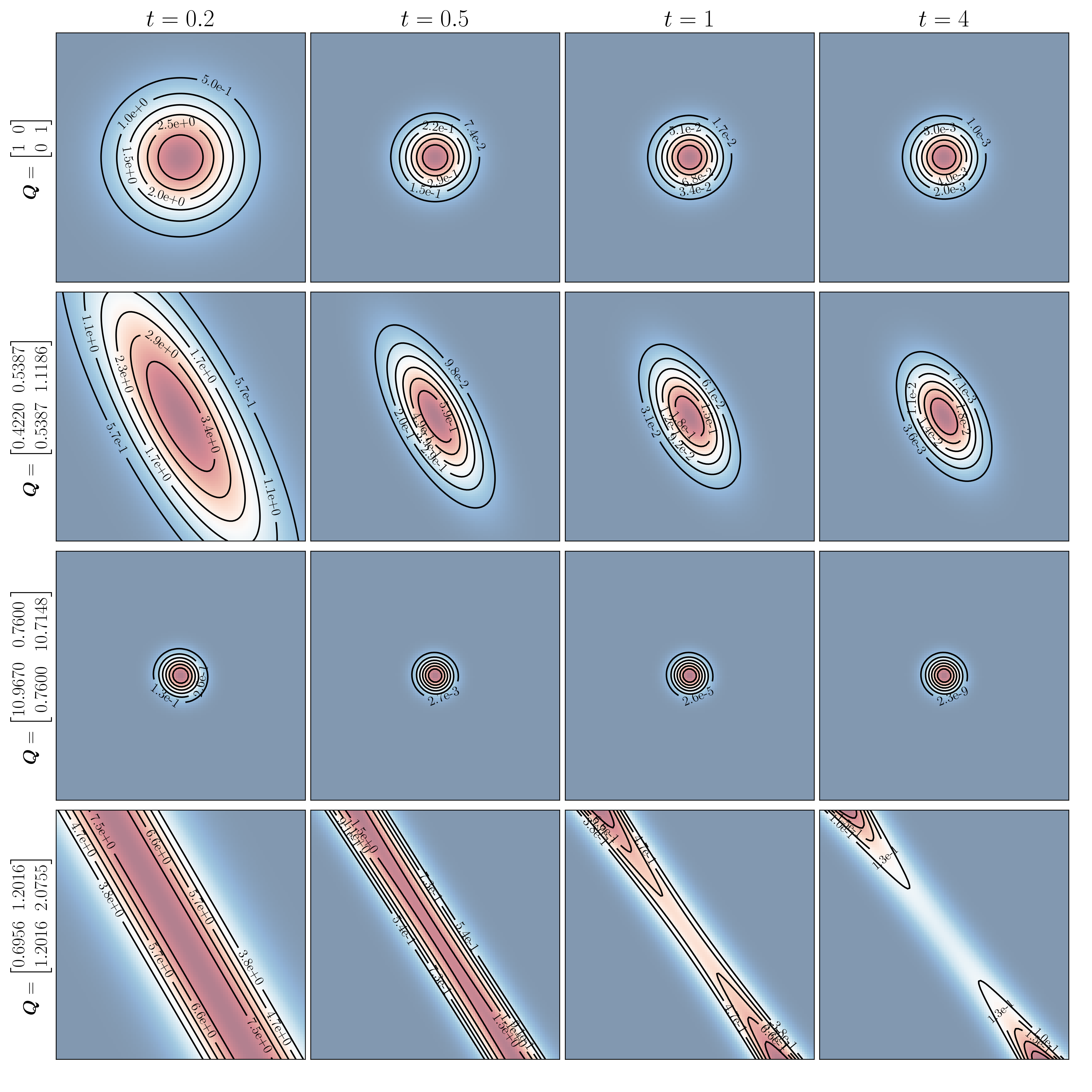}
        \caption{Solution to PDE \eqref {FactorPDEForward} in $n=2$ dimensions with $t_0 = 0$ and initial condition $\widehat{\varphi}_{0}(\cdot)=1$. All subfigures are over the spatial domain $\left[-7,7\right]^{2}$.}
\vspace*{-0.1in}	\label{fig:constant_initial_condition} 
\end{figure}


\section{Discussions}\label{sec:Discussions}
In this section, we discuss various implications of the results derived so far. We start by recording the basic properties of the kernel $\kappa_{+}$ and of the corresponding solution \eqref{PhiHatAsIntegral}.
\begin{proposition}[\textbf{Properties of the kernel and the solution}]\label{Prop:Properties}
     For $0\leq t_{0} < t < \infty$, given $\bm{Q}\succeq\bm{0}$, $\widehat{\varphi}_{0}(\cdot)$ Lebesgue measurable and everywhere positive, the kernel $\kappa_{+}$ and the solution \eqref{PhiHatAsIntegral}  satisfy:
    \begin{enumerate}
    \item[(i)] (\textbf{spatial symmetry}) $\kappa_{+}(t_0, \bm{y}, t, \bm{z}) =\kappa_{+}(t_0, \bm{z}, t, \bm{y})$,

    \item[(ii)] (\textbf{positivity}) $\widehat{\varphi}(t,\bm{x})>0$ for all $\bm{x}\in\mathbb{R}^{n}$, 
    
    \item[(iii)] (\textbf{asymptotic decay}) $\lim_{|\bm{x}|\rightarrow\infty} \widehat{\varphi}(t, \bm{x} )=\lim_{t\rightarrow\infty} \widehat{\varphi}(t, \bm{x} )=0$.
    
\end{enumerate}
\end{proposition}
\begin{proof}
(i) Follows from the spatial symmetries of the kernels \eqref{defHeatKernel} and \eqref{kernel_function}.

\noindent(ii) Since \eqref{defHeatKernel} and \eqref{kernel_function} are both positive, so is $\kappa_{+}$ in \eqref{semi_definite_kernel}. Then, $\widehat{\varphi}_{0}$ being everywhere positive, \eqref{PhiHatAsIntegral} is positive too.

\noindent(iii) From Proposition \ref{Proposition:MatrixM}(iii), $\bm{M}\succ\bm{0}$ and hence the exponential term in \eqref{kernel_function} has negative definite argument. The same is true for the heat kernel \eqref{defHeatKernel}. Thus the product of the exponential terms in \eqref{semi_definite_kernel} is bounded in $\left[0,1\right]$. So $\vert \kappa_{+}(t_{0}, \bm{Vx}, t, \bm{z}) \:\widehat{\varphi}_0\left(\bm{V}^{\top}\bm{z}\right)\vert = \kappa_{+}(t_{0}, \bm{Vx}, t, \bm{z}) \:\widehat{\varphi}_0\left(\bm{V}^{\top}\bm{z}\right) \leq \zeta(t,t_0)\widehat{\varphi}_0\left(\bm{V}^{\top}\bm{z}\right)$, where $\zeta(t,t_0):=(4\pi(t-t_0))^{-p/2}\allowbreak(2\pi)^{-(n-p)/2}\left(\det(\bm{D}_{\left[i_{1}:i_{n-p}\right]})\right)^{1/4}/\sqrt{\det\left(\sinh\left(2(t-t_0)\sqrt{\bm{D}_{\left[i_{1}:i_{n-p}\right]}}\right)\right)}$. Because $\widehat{\varphi}_{0}$ is Lebesgue measurable (given), applying the dominated convergence theorem separately for $\vert\bm{x}\vert\rightarrow\infty$ with finite $t$, and for $t\rightarrow\infty$, concludes the proof.\qedsymbol
\end{proof}
\begin{figure}[t]
	\centering
        \includegraphics[
       width=\linewidth]{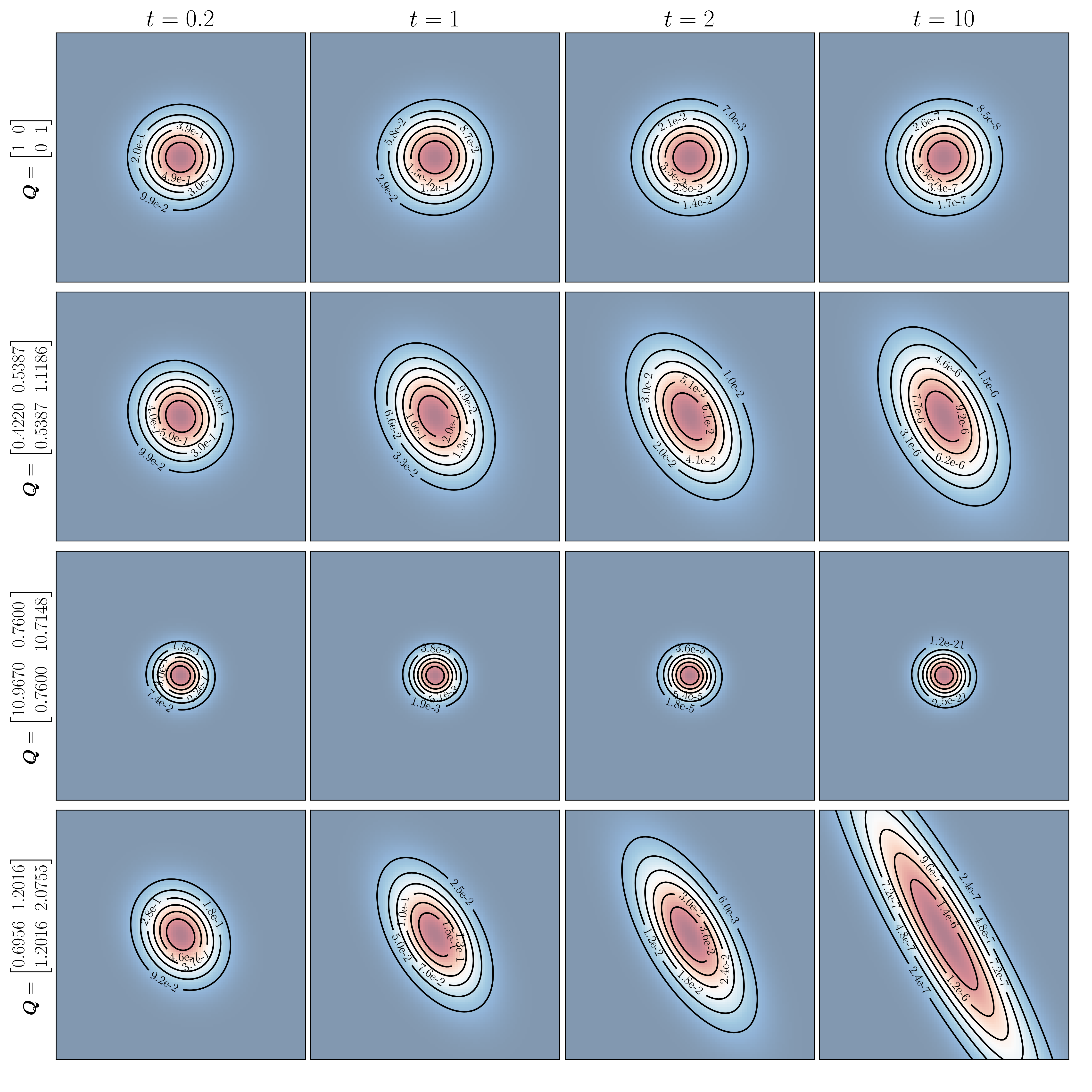}
        \caption{Solution to PDE \eqref {FactorPDEForward} in $n=2$ dimensions with $t_0 = 0$ and initial condition $\widehat{\varphi}_{0}(\cdot)=\frac{1}{ 2\pi}e^{-\frac{1}{2}\|\cdot\|^2}$. All subfigures are over the spatial domain $\left[-5,5\right]^{2}$.}
        \label{fig:normal_distribution_initial_condition} 
\end{figure}

\begin{figure}[t]
	\centering
        \includegraphics[
       width=\linewidth]{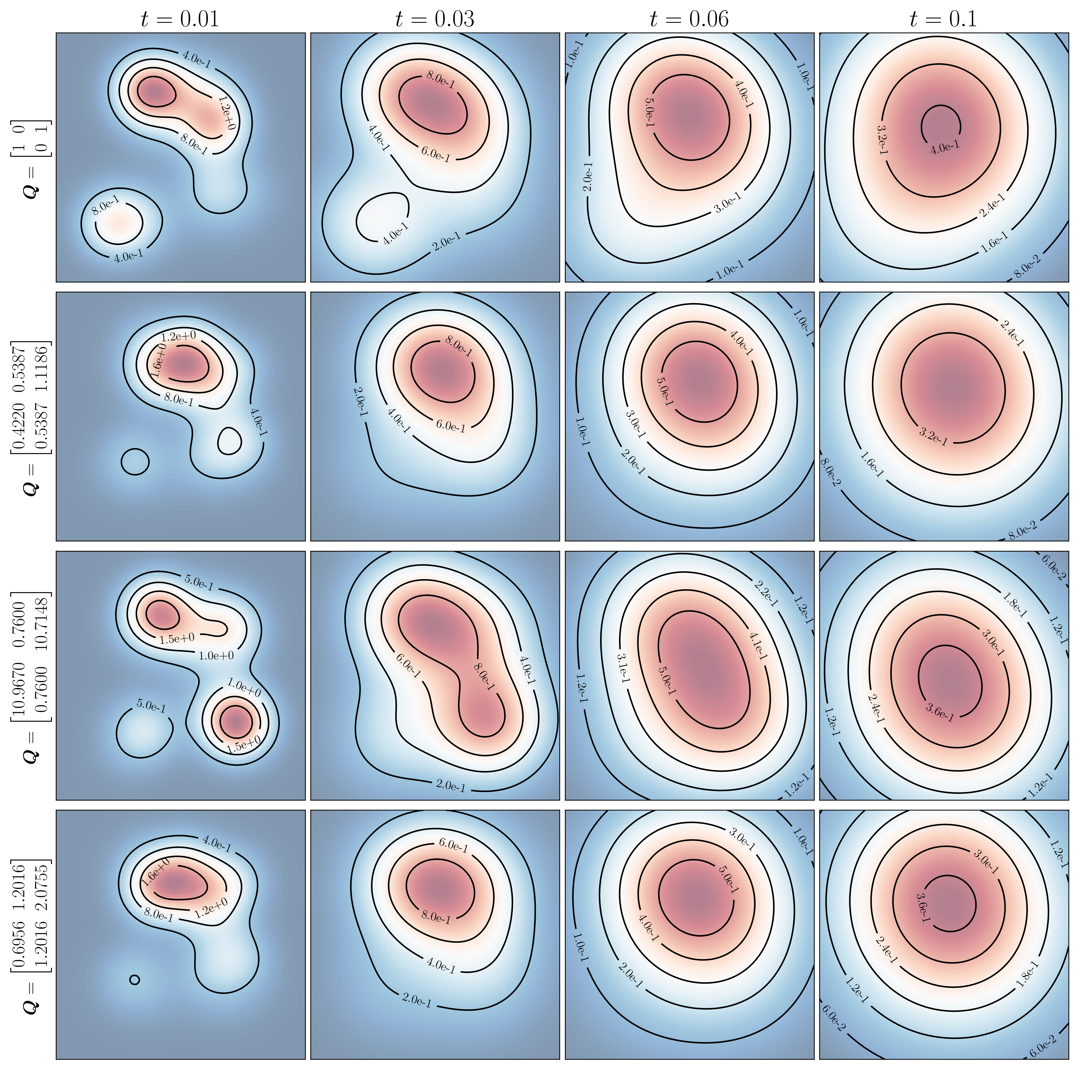}
        \caption{Solution to PDE \eqref {FactorPDEForward} in $n=2$ dimensions with $t_0 = 0$ and initial condition $\widehat{\varphi}_{0}(\cdot)=\rho_{\infty} $. All subfigures are over the spatial domain $\left[-0.25,1.25\right]^{2}$.}
        \label{fig:non_gaussian_initial_condition} 
\end{figure}
In Fig. \ref{fig:constant_initial_condition}, Fig. \ref{fig:normal_distribution_initial_condition} and Fig. \ref{fig:non_gaussian_initial_condition}, we illustrate the solution $\widehat{\varphi}(t,\bm{x})$ in \eqref{PhiHatAsIntegral} for $n=2$ dimensions with $t_0=0$, initial conditions $\widehat{\varphi}_{0}(\cdot)=1$, $\widehat{\varphi}_{0}(\cdot)=\mathcal{N}(\bm{0},\bm{I})$, {\black{and $\widehat{\varphi}_{0}(\cdot)\propto\exp\left(-f(x_1,x_2)/35\right)$ where $f(x_1,x_2):=\left ( (10x_1-5)^2+10x_2-16\right )^2+\left ( 10x_1-12+(10x_2-5)^2 \right)^2$, $(x_1,x_2)\in[0,1]^2$, respectively. We point out that the last $\widehat{\varphi}_0$ is of non-Gaussian Gibbs-type \cite[Sec. V-C]{halder2020hopfield}, and $f$ there is a scaled Himmelblau function, often used as a benchmark in nonconvex optimization \cite{himmelblau2018applied}.}} 

In these figures, the columns correspond to different times and the rows correspond to different problem data: $\bm{Q}\succeq \bm{0}$. These results agree with the properties in Proposition \ref{Prop:Properties}, and were obtained by using the derived kernels in Theorem \ref{Thm:Kernel} and Theorem \ref{SemiPosDefKernel}. For the three cases of $\widehat{\varphi}_{0}$ reported here, we analytically performed the integration in \eqref{PhiHatAsIntegral} using Lemma \ref{Lemma:CentralIdentityQFT}.

Specifically, for $t_0=0$, $\widehat{\varphi}_{0}(\cdot)=1$ and $\bm{Q}\succ\bm{0}$, \eqref{EtaHatWithPosDefKernel}-\eqref{kernel_function} give
\begin{align}
&\widehat{\eta}(t,\bm{y})\big\vert_{\widehat{\varphi}_{0} = 1} = \frac{(\det(\bm{D}))^{1/4}}{(2\pi)^{\frac{n}{2}}\sqrt{\det(\sinh(2t\sqrt{\bm{D}}))}} \exp\!\left(\! 
 -\frac{1}{2} \bm{y}^{\top} \bm{D}^{\frac{1}{2}}\coth\left(2t\sqrt{\bm{D}}\right)\bm{y}\!\right)\label{etahatICunity}\\
 &\!\int_{\mathbb{R}^n}\!\exp\!\left(\!-\frac{1}{2}\bm{z}^{\top}\bm{D}^{\frac{1}{2}}\coth\left(2t\sqrt{\bm{D}}\right)\bm{z}  +\bm{y}^{\top} \bm{D}^{\frac{1}{2}}\text{csch}\left(2t\sqrt{\bm{D}}\right)\bm{z}\!\right)\differential\bm{z}\nonumber\\
&= \dfrac{1}{\sqrt{\det\left(\cosh\left(2t\sqrt{\bm{D}}\right)\right)}}\!\exp\!\left(\!-\frac{1}{2}\bm{y}^{\top}\bm{D}^{\frac{1}{2}}\tanh\left(2t\sqrt{\bm{D}}\right)\bm{y}\!\right),\nonumber
\end{align}
where the last equality follows from invoking Lemma \ref{Lemma:CentralIdentityQFT} for the integral in the preceding line. The corresponding 
\begin{align}
\widehat{\varphi}(t,\bm{x})\big\vert_{\widehat{\varphi}_{0} = 1} = \widehat{\eta}\left(t,\bm{y}=\bm{V}^{\top}\bm{x}\right)\big\vert_{\widehat{\varphi}_{0} = 1} &= \dfrac{\!\exp\!\left(\!-\frac{1}{2}\bm{x}^{\top}\bm{V}\bm{D}^{\frac{1}{2}}\tanh\left(2t\sqrt{\bm{D}}\right)\bm{V}^{-1}\bm{x}\!\right)}{\sqrt{\det\left(\cosh\left(2t\sqrt{\bm{D}}\right)\right)}}\label{phihatICunity}\\
&=\dfrac{\!\exp\!\left(\!-{\frac{1}{4}}\bm{x}^{\top}\sqrt{2\bm{Q}}\tanh\left(t\sqrt{2\bm{Q}}\right)\bm{x}\!\right)}{\sqrt{\det\left(\cosh\left(t\sqrt{2\bm{Q}}\right)\right)}},
\nonumber 
\end{align}
where the last equality used \eqref{MatrixFunction}. For sanity check, we verify from \eqref{phihatICunity} that $\widehat{\varphi}(t=0,\bm{x})\big\vert_{\widehat{\varphi}_{0} = 1} = 1$, as expected. Also, for $\bm{D}\downarrow\bm{0}$ (heat kernel limit), \eqref{phihatICunity} returns unity which is intuitive since the heat kernel leaves constant function invariant. For the results in Fig. \ref{fig:constant_initial_condition}, we used \eqref{phihatICunity} to illustrate the effect of varying $\bm{Q}\succeq\bm{0}$. In particular, the last row of Fig. \ref{fig:constant_initial_condition} used Theorem \ref{SemiPosDefKernel} to account for a zero eigenvalue. That row shows the joint effect of one eigen-direction evolving toward uniform (diffusion due to $\kappa_0$) and another eigen-direction evolving as per $\kappa_{++}$.

\begin{figure}[t]
  \centering
  \includegraphics[width=0.56\textwidth]{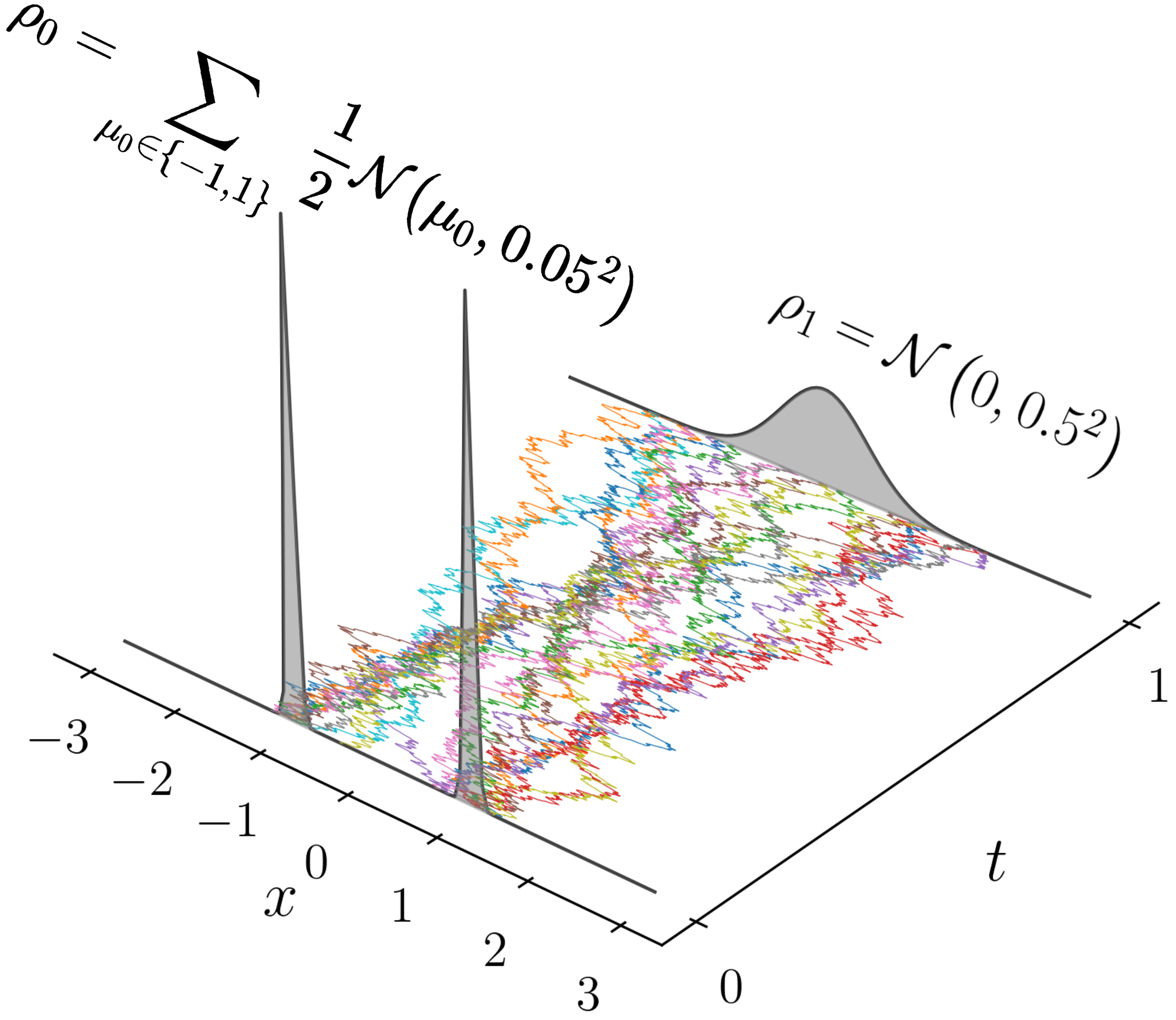}
  \caption{{\small{Optimally controlled sample paths for 1D SB with state cost $\frac{1}{2}Qx^2$, $Q=2$, shown with endpoint endpoint PDFs $\rho_0 = \sum_{\mu_0\in\{-1,1\}}\frac{1}{2}\mathcal{N}(\mu_0,0.05^2)$, $\rho_1=\mathcal{N}(0,0.5^2)$ and time horizon $[t_0,t_1]=[0,1]$.}}}
  \label{Fig:1dSBPregualrized3dstacked}
\end{figure}

Likewise, for $t_0=0$, $\widehat{\varphi}_{0}(\cdot)=\mathcal{N}\left(\bm{0},\bm{I}\right)$ and $\bm{Q}\succ\bm{0}$, \eqref{EtaHatWithPosDefKernel}-\eqref{kernel_function} give
\begin{align}
&\widehat{\eta}(t,\bm{y})\big\vert_{\widehat{\varphi}_{0} = \mathcal{N}\left(\bm{0},\bm{I}\right)} = \frac{(\det(\bm{D}))^{1/4}}{(2\pi)^{\frac{n}{2}}\sqrt{\det(\sinh(2t\sqrt{\bm{D}}))}} \exp\!\left(\! 
 -\frac{1}{2} \bm{y}^{\top} \bm{D}^{\frac{1}{2}}\coth\left(2t\sqrt{\bm{D}}\right)\bm{y}\!\right)\label{etahatICstandardnormal}\\
 &\!\int_{\mathbb{R}^n}\!\exp\!\left(\!-\frac{1}{2}\bm{z}^{\top}\bigg\{\bm{D}^{\frac{1}{2}}\coth\left(2t\sqrt{\bm{D}}\right)+\bm{I}\bigg\}\bm{z}  +\bm{y}^{\top} \bm{D}^{\frac{1}{2}}\text{csch}\left(2t\sqrt{\bm{D}}\right)\bm{z}\!\right)\differential\bm{z}\nonumber\\
&=\hspace{-3pt}\frac{(\det{\bm{D}})^{\frac{1}{4}}}{\sqrt{\det\left(\sqrt{\bm{D}}\cosh\left(2t\sqrt{\bm{D}}\right)+\sinh\left(2t\sqrt{\bm{D}}\right)\right)}}\!\exp\!\left(\!-\frac{1}{2}\bm{y}^{\top}\bm{D}^{\frac{1}{2}} \bigg\{\coth\left(2t\sqrt{\bm{D}}\right)\right.\nonumber\\
   &\hspace*{29pt}\left. -\:{\rm{csch}}\left(2t\sqrt{\bm{D}}\right)\left(\sqrt{\bm{D}}\coth\left(2t\sqrt{\bm{D}}\right)+\bm{I}\right)^{-1}\sqrt{\bm{D}}\:{\rm{csch}}\left(2t\sqrt{\bm{D}}\right)\bigg\}\bm{y}\right),\nonumber
\end{align}
where, once again, the last equality follows from invoking Lemma \ref{Lemma:CentralIdentityQFT} for the integral in the preceding line. The corresponding 
\begin{align}
&\widehat{\varphi}(t,\bm{x})\big\vert_{\widehat{\varphi}_{0} = \mathcal{N}\left(\bm{0},\bm{I}\right)} = \widehat{\eta}\left(t,\bm{y}=\bm{V}^{\top}\bm{x}\right)\big\vert_{\widehat{\varphi}_{0} = \mathcal{N}\left(\bm{0},\bm{I}\right)}\label{phihatICstandardnormal}\\
   &=\hspace{-3pt}{\black{\frac{2^{-n/4}(\det\bm{Q})^{\frac{1}{4}}}{\sqrt{\det\left(\sqrt{2\bm{Q}}\cosh\left(t\sqrt{2\bm{Q}}\right)/2 +\sinh\left(t\sqrt{2\bm{Q}}\right)\right)}}\!\exp\!\left(\!-\frac{1}{4}\bm{x}^{\top}\sqrt{2\bm{Q}} \bigg\{\coth\left(t\sqrt{2\bm{Q}}\right)\right.}}\nonumber\\
   &\hspace*{29pt}{\black{\left. -\:{\rm{csch}}\left(t\sqrt{2\bm{Q}}\right)\left(\frac{1}{4}\sqrt{2\bm{Q}}\coth\left(t\sqrt{2\bm{Q}}\right)+\bm{I}\right)^{-1}\sqrt{2\bm{Q}}\:{\rm{csch}}\left(t\sqrt{2\bm{Q}}\right)\bigg\}\bm{x}\right)}},\nonumber
\end{align}
where the last equality is due to \eqref{MatrixFunction}. For the results in Fig. \ref{fig:normal_distribution_initial_condition}, we used \eqref{phihatICstandardnormal} to illustrate the effect of varying $\bm{Q}\succeq\bm{0}$. As before, the last row of Fig. \ref{fig:normal_distribution_initial_condition} used Theorem \ref{SemiPosDefKernel} to account for a zero eigenvalue.

{\black{Likewise, the results in Fig. \ref{fig:non_gaussian_initial_condition} show the transient action of the kernel on non-Gaussian $\widehat{\varphi}_0$.}}

We conclude with a simulation result that brings together the ideas to solve the SB problem with state cost-to-go. In Fig. \ref{Fig:1dSBPregualrized3dstacked}, we show the closed loop optimally controlled sample paths together with endpoint PDFs $\rho_0 = \sum_{\mu_0\in\{-1,1\}}\frac{1}{2}\mathcal{N}(\mu_0,0.05^2)$, $\rho_1=\mathcal{N}(0,0.5^2)$ for an 1D SB with state cost $\frac{1}{2}Qx^2$ with $Q=2$ and time horizon $[t_0,t_1]=[0,1]$. This result is obtained by performing the dynamic Sinkhorn recursions shown in Fig. \ref{fig:FixedPointRecursion} wherein we solved the forward-backward IVPs for the PDEs \eqref{FactorPDEs} as in Remark \ref{remark:FromEtaHatBackToPhiHat}, which was in turn made possible by the kernel $\kappa_{++}$ deduced in Theorem \ref{Thm:Kernel}.

\noindent\textbf{Future directions.} We mention several directions in which our results could be extended. It will be of interest to extend our results for time-varying weights in quadratic state-cost-to-go, i.e., by considering matricial trajectory $\bm{Q}(t)$ that is continuous and bounded in $t$, and $\succeq \bm{0}$ for each $t\in[t_0,t_1]$. Such time-varying weights are straightforward to account for in classical LQ optimal control but it is not apparent how to generalize our Hermite polynomial-related developments for this case. From a control-theoretic viewpoint, it will also be of interest to solve problem \eqref{LambertSBP} with the same quadratic state cost-to-go as in this work but with prior dynamics $\differential\bm{x}_t=\bm{A}(t)\bm{x}_t \differential t + \sqrt{2}\bm{B}(t)\differential\bm{w}_t$ as opposed to the case $\left(\bm{A}(t),\bm{B}(t)\right)=\left(\bm{0},\bm{I}\right)$ that we considered here following classical SB. {\black{In a different vein, it will be of interest to find the worst-case contraction coefficient \cite{teter2023contraction} for the recursion in Fig. \ref{fig:FixedPointRecursion} as function of the matrix $\bm{Q}$, using the derived kernel.}} These will be explored in our future work.

\appendix
\section{Proofs}\label{AppSec:Proofs}

\subsection{Proof for Theorem \ref{ThmLSBPExistenceUniqueness}} \label{Proof_of_ThmLSBPExistenceUniqueness}
The main idea is to derive a stochastic calculus of variations formulation, i.e., an analogue of \eqref{KLmin} for measurable nonzero $q$. To this end, consider $\mathbb{P}\in\mathcal{M}(\Omega)$ generated by the It\^{o} diffusion $\differential\bm{x}_t^{\bm{u}}=\bm{u}(t,\bm{x}_t^{\bm{u}})\differential t + \sqrt{2}\differential\bm{w}_t$. Recall that $\bm{u}\in\mathcal{U}$, the space of finite energy Markovian controls, i.e.,
\begin{align}
\mathcal{U}:= \{\bm{u}:[t_0,t_1]\times\mathbb{R}^{n}\mid \vert \bm{u}\vert^2 < \infty\}.
\label{FeasibleControls}    
\end{align}

Letting $\mathbb{P}_{0},\mathbb{W}_{0}$ be the distributions of $\bm{x}_0^{\bm{u}}$ under $\mathbb{P}$ and $\mathbb{W}$ respectively, we use Girsanov's theorem \cite[Thm. 8.6.6]{oksendal2013stochastic}, \cite[Ch. 3.5]{karatzas2012brownian} to find that the Radon-Nikodym derivative $\dfrac{\differential\mathbb{P}}{\differential\mathbb{W}} = \dfrac{\differential\mathbb{P}_{0}}{\differential\mathbb{W}_{0}} \exp\left(\displaystyle\int_{t_0}^{t_1}\dfrac{\bm{u}}{\sqrt{2}}\differential\bm{w}_t + \displaystyle\int_{t_0}^{t_1}\dfrac{1}{2}\dfrac{\vert\bm{u}\vert^2}{2}\differential t\right)$ $\mathbb{P}$ a.s. Therefore,
\begin{align}
\mathbb{E}_{\mathbb{P}}\left[\log\dfrac{\differential\mathbb{P}}{\differential\mathbb{W}}\right] = \mathbb{E}_{\mathbb{P}}\left[\log\dfrac{\differential\mathbb{P}_{0}}{\differential\mathbb{W}_{0}}\right] +\dfrac{1}{\sqrt{2}}\mathbb{E}_{\mathbb{P}}\left[\displaystyle\int_{t_0}^{t_1}\bm{u}\differential\bm{w}_t\right] + \dfrac{1}{4}\displaystyle\int_{t_0}^{t_1}\mathbb{E}_{\mathbb{P}}\left[\vert\bm{u}\vert^2\right]\differential t,
\label{ExpectedLog}
\end{align}
where we used the Fubini-Tonelli theorem for the last term.

Now we argue that the second summand in the RHS of \eqref{ExpectedLog} equals zero. This is because we know from \eqref{FeasibleControls} that $\mathbb{E}_{\mathbb{P}}\left[\int_{t_0}^{t_1}\vert\bm{u}\vert^2\differential t\right] < \infty$, so for any $t\geq t_{0}$, the process $\int_{t_0}^{t}\bm{u}\differential w_{\tau}$ is a martingale with constant expected value $\mathbb{E}_{\mathbb{P}}\left[\int_{t_0}^{t}\bm{u}\differential w_{\tau}\right] = \mathbb{E}_{\mathbb{P}}\left[\int_{t_0}^{t_0}\bm{u}\differential w_{\tau}\right]=0$. Thus, \eqref{ExpectedLog} simplifies to
\begin{align}
\mathbb{E}_{\mathbb{P}}\left[\log\dfrac{\differential\mathbb{P}}{\differential\mathbb{W}}\right] = \mathbb{E}_{\mathbb{P}}\left[\log\dfrac{\differential\mathbb{P}_{0}}{\differential\mathbb{W}_{0}}\right] + \dfrac{1}{4}\displaystyle\int_{t_0}^{t_1}\mathbb{E}_{\mathbb{P}}\left[\vert\bm{u}\vert^2\right]\differential t.
\label{ExpectedLogSimplified}
\end{align}

To account for possibly unbounded $q(t,\cdot)$, define the positive and negative parts of $q(t,\cdot)$ as $q^{+}(t,\cdot):=\max\{q(t,\cdot),0\}$ and $q^{-}(t,\cdot):=\max\{-q(t,\cdot),0\}$, respectively. Define stopping time $\tau_t := \inf\{\widetilde{t}>t \mid \int_{t}^{\tilde{t}}q^{-}(s,\cdot)\differential s=\infty\}$ with the convention that $\tau_t=\infty$ when there is no such $\widetilde{t}$. Following \cite[Sec. 4]{aebi1992large}, we next define $\mathcal{X}:=\{(t,\bm{x}_t^{\bm{u}})\in[t_0,t_1]\times\mathbb{R}^{n}\mid \mathbb{P}(t_1<\tau_t)>0\}$, and for any nonnegative measurable $F(\cdot)$, let
\begin{align*}
\mathbb{P}^{-}_{t}(F)&:=\mathbb{P}\left[\exp\left(\frac{1}{2}\int_{t}^{t_1}\!\!q^{-}(s,\bm{x}^{\bm{u}}_{s})\differential s\!\right)F(\cdot)\bm{1}_{\{t_{1}<\tau_{t}\}}\right] \nonumber\\
&= \!\int\!\exp\left(\frac{1}{2}\int_{t}^{t_1}\!\!\!q^{-}(s,\bm{x}^{\bm{u}}_{s})\differential s\!\right)\!F(\cdot)\bm{1}_{\{t_{1}<\tau_{t}\}}\differential\mathbb{P}.
\end{align*}
Clearly, $\mathbb{P}_{t}^{-}(1)>0$ $\forall(t,\bm{x}_t^{\bm{u}})\in\mathcal{X}$. Assuming $\mathbb{P}_{t}^{-}\left(\exp\left(-\frac{1}{2}\int_{t}^{t_1}q^{+}(s,\bm{x}_s^{\bm{u}})\differential s\right)\right) \allowbreak= \mathbb{P}\left(\exp\left(-\frac{1}{2}\int_{t}^{t_1}q(s,\bm{x}_s^{\bm{u}})\differential s\right)\bm{1}_{\{t_1<\tau_t\}}\allowbreak\right)\allowbreak< \infty$
$\forall t$ and a.e. $\bm{x}_t^{\bm{u}}$ in $\mathcal{X}$, we then define the reweighing measure
\begin{align*}
\mathbb{P}^{q}(F) &:= \mathbb{P}^{-}_{t_0}\left(\exp\left(-\dfrac{1}{2}\int_{t_{0}}^{t_1}q^{+}(s,\bm{x}_s^{\bm{u}})\differential s\right)F(\cdot)\right)\\
&= \mathbb{P}\left(\exp\left(-\dfrac{1}{2}\int_{t_{0}}^{t_1}q(s,\bm{x}_s^{\bm{u}})\differential s\right)\bm{1}_{\{t_1<\tau_{t_{0}}\}}F(\cdot)\right).
\end{align*}
Taking $\mathbb{P}^{q}(1)\mathbb{W}/Z\in\mathcal{M}(\Omega)$ as the reference measure, where the normalization constant $Z:= \int_{\Omega}\differential\left(\mathbb{P}^{q}(1)\mathbb{W}\right)$, we then find 
\begin{align}
&D_{\rm{KL}}\left(\mathbb{P}\parallel\mathbb{P}^{q}(1)\mathbb{W}/Z\right) = \log Z + \mathbb{E}_{\mathbb{P}}\!\left[\log\dfrac{\differential\mathbb{P}}{\differential\mathbb{W}}\right] + \mathbb{E}_{\mathbb{P}}\!\left[\!\left(\dfrac{1}{2}\int_{t_0}^{t_1}\!\!\!q(s,\bm{x}_s^{\bm{u}})\differential s\right)\!\bm{1}_{\{t_1<\tau_{t_{0}}\}}\!\right].
\label{KLdivRefMeasure}
\end{align}
Using \eqref{ExpectedLogSimplified}, the identity \eqref{KLdivRefMeasure} simplifies to 
\begin{align}
&D_{\rm{KL}}\left(\mathbb{P}\parallel\mathbb{P}^{q}(1)\mathbb{W}/Z\right) \nonumber\\
&=\log Z + \mathbb{E}_{\mathbb{P}}\!\left[\log\dfrac{\differential\mathbb{P}_0}{\differential\mathbb{W}_0}\right] + \frac{1}{2}\int_{t_0}^{t_1}\mathbb{E}_{\mathbb{P}}\left[\frac{1}{2}\vert\bm{u}\vert^2 + q(t,\bm{x}_t^{\bm{u}})\bm{1}_{\{t_1<\tau_{t_{0}}\}}\right]\differential t.
\label{SBbbjectiveAsKL}
\end{align}
Recall $\Pi_{01}$ as in Sec. \ref{sec:Intro}, and notice that the term $\log Z + \mathbb{E}_{\mathbb{P}}\!\left[\log\dfrac{\differential\mathbb{P}_0}{\differential\mathbb{W}_0}\right]$ appearing in the RHS of \eqref{SBbbjectiveAsKL} is constant over $\Pi_{01}$. Therefore, the $\arg\inf$ of $$\int_{t_0}^{t_1}\mathbb{E}_{\mathbb{P}}\left[\frac{1}{2}\vert\bm{u}\vert^2 + q(t,\bm{x}_t^{\bm{u}})\bm{1}_{\{t_1<\tau_{t_{0}}\}}\right]\differential t,$$
is the same as that of $D_{\rm{KL}}\left(\mathbb{P}\parallel\mathbb{P}^{q}(1)\mathbb{W}/Z\right)$ over $\Pi_{01}$.

However, the map $\mathbb{P}\mapsto D_{\rm{KL}}\left(\mathbb{P}\parallel\mathbb{Q}\right)$ is strictly convex for any fixed $\mathbb{Q}$. So the existence-uniqueness for the $\arg\inf$ are guaranteed. Notice that the conclusion holds irrespective of $\bm{1}_{\{t_1<\tau_{t_{0}}\}}$ being zero or one. Also note that if $q(t,\cdot)$ is bounded for all $t\in[t_0,t_1]$, then the proof can be simplified by eschewing the stopping time related constructs.\qedsymbol


\subsection{Proof for Theorem \ref{ThmHopfColeReactionDiffusion}} \label{Proof_of_ThmHopfColeReactionDiffusion}
Denoting $\psi_{t}\in\mathcal{C}^{1,2}\left([t_0,t_1];\mathbb{R}^{n}\right)$ to be suitable Lagrange multiplier, the Lagrangian for problem \eqref{LambertSBP} is the functional 
\begin{align}
    \int_{\mathbb{R}^n}\int_{t_0}^{t_1}\left( \left(\frac{1}{2} \vert \bm{u}\vert^2 + q(\bm{x}_t^{\bm{u}}) \right) \rho_t + \psi_t\times \left( \frac{\partial \rho_t}{\partial t} + \nabla \cdot (\rho_t \bm{u}) - \Delta \rho_t \right) \right)\differential t\:\differential\bm{x}_t^{\bm{u}}.
    \label{Langrangian_orginal}
\end{align}
We apply the Fubini-Tonelli Theorem to switch the order of the integration and integrate by parts w.r.t. $t$ to obtain
\begin{align*}
    \int_{t_0}^{t_1} \int_{\mathbb{R}^{n}} \psi_t \frac{\partial \rho_t}{\partial t}\:\differential\bm{x}_t^{\bm{u}}\:\differential t = \int_{\mathbb{R}^{n}} \left(\psi_{1}(\cdot)\rho_{1}(\cdot)-\psi_{0}(\cdot)\rho_0(\cdot)\right)\differential(\cdot) - \int_{t_0}^{t_1} \int_{\mathbb{R}^{n}}\rho_t \frac{\partial \psi_t}{\partial t}\:\differential\bm{x}_t^{\bm{u}}\:\differential t.
\end{align*}
Assuming the limits at $|\bm{x}^{\bm{u}}_{t}|\rightarrow\infty$ are zero, and integrating by parts w.r.t. $\bm{x}_t^{\bm{u}}$, we have
\begin{align*}
&\int_{t_0}^{t_1} \int_{\mathbb{R}^{n}}\psi_t\left(\nabla\cdot (\rho_t \bm{u}) - \Delta \rho_t\right)\differential\bm{x}_t^{\bm{u}}\:\differential t \nonumber\\
&= -\int_{t_0}^{t_1} \int_{\mathbb{R}^{n}}\langle\nabla\psi_t, \bm{u}\rangle\rho_t\differential\bm{x}_t^{\bm{u}}\differential t + \int_{t_0}^{t_1} \int_{\mathbb{R}^{n}}\langle\nabla\rho_t,\nabla\psi_t\rangle\differential\bm{x}_t^{\bm{u}}\differential t \\ 
&= -\int_{t_0}^{t_1} \int_{\mathbb{R}^{n}}\langle\nabla\psi_t, \bm{u}\rangle\rho_t\differential\bm{x}_t^{\bm{u}}\differential t - \int_{t_0}^{t_1} \int_{\mathbb{R}^{n}}\rho_t\Delta\psi_t \differential\bm{x}_t^{\bm{u}}\:\differential t.
\end{align*}

We can therefore rewrite \eqref{Langrangian_orginal} as
\begin{align}               \int_{t_0}^{t_1}\int_{\mathbb{R}^{n}}\left(\dfrac{1}{2}\vert\bm{u}\vert^{2} + q(\bm{x}_t^{\bm{u}}) - \dfrac{\partial\psi_t}{\partial t} -  \langle\nabla\psi_t,\bm{u}\rangle - 
\Delta\psi_t \right) \rho_t \:\differential\bm{x}_t^{\bm{u}}\:\differential t.
\label{Lagrangian_new}
\end{align}
By pointwise minimization of \eqref{Lagrangian_new} with respect to $\bm{u}$, we determine that the optimal control is 
\begin{align}
    \bm{u}^{\rm{opt}} = \nabla \psi_t.
    \label{inputSBP}
\end{align}
Evaluating \eqref{Lagrangian_new} at \eqref{inputSBP} and setting the resulting integral equal to zero, yields
\begin{align}
\int_{t_0}^{t_1}\int_{\mathbb{R}^{n}}\left(- \dfrac{1}{2}\vert\nabla \psi_t \vert^{2} + q(\bm{x}_t^{\bm{u}}) - \dfrac{\partial\psi_t}{\partial t} - \Delta\psi_t \right) \rho^{\rm{opt}}_t \:\differential\bm{x}_t^{\bm{u}}\:\differential t = 0.
\label{dynammic_programming}
\end{align}
For \eqref{dynammic_programming} to hold for an arbitrary $\rho_0$, and thus for arbitrary $\rho^{\rm{opt}}_t$, we must have
\begin{align*}
    - \dfrac{1}{2}\vert\nabla \psi_t \vert^{2} + q(\bm{x}_t^{\bm{u}}) - \dfrac{\partial\psi_t}{\partial t} -  \Delta\psi_t=0, 
\end{align*}
and so we arrive at a Hamilton-Jacobi-Bellman (HJB)-like PDE
\begin{align*}
    \dfrac{\partial\psi_t}{\partial t} +\dfrac{1}{2}\vert \nabla \psi_t \vert^{2} + \Delta\psi_t = q(\bm{x}_t^{\bm{u}}).
\end{align*}
Combining \eqref{LambertSBPdynconstr} and \eqref{inputSBP} yields \begin{align*}
    \dfrac{\partial\rho^{\rm{opt}}_t}{\partial t} + \nabla\cdot \left(\rho_t^{\rm{opt}}\nabla \psi_t\right)= \Delta\rho_t^{\rm{opt}}.
\end{align*}

We have therefore shown that the tuple $\left(\rho^{\rm{opt}}_t,\bm{u}^{\rm{opt}}\right)$ solving problem \eqref{LambertSBP} satisfies the system of coupled PDEs:
\begin{subequations}
    \begin{align}
        &\dfrac{\partial\psi_t}{\partial t} +\dfrac{1}{2}\vert \nabla \psi_t \vert^{2} + \Delta\psi_t = q(\bm{x}_t^{\bm{u}}),\label{HJBSBP}\\
         &\dfrac{\partial\rho^{\rm{opt}}_t}{\partial t} + \nabla\cdot \left(\rho_t^{\rm{opt}}\nabla \psi_t\right)=\Delta\rho_t^{\rm{opt}}, \label{FPKSBP}
    \end{align}
    \label{FirstOrderConditions}
\end{subequations}
with boundary conditions
\begin{align}
    \rho^{\rm{opt}}_t(t=t_0,\cdot) = \rho_0(\cdot), \quad \rho_t^{\rm{opt}}(t=t_1,\cdot) = \rho_1(\cdot).
    \label{BoundaryConditionsForProp1}
\end{align}

To convert the coupled nonlinear PDEs \eqref{FirstOrderConditions} with decoupled boundary conditions \eqref{BoundaryConditionsForProp1}, to a system of decoupled linear reaction-diffusion PDEs with coupled boundary conditions, we consider the Hopf-Cole transform \cite{hopf1950partial,cole1951quasi} given by 
\begin{subequations}
\begin{align}
\varphi &= \exp \left( \psi_t \right), \label{phi_epsilon}\\
\widehat{\varphi} &= \rho^{\rm{opt}}_t \exp \left( - \psi_t \right).\label{phi_hat_epsilon}
\end{align} 
\label{defphiphihat}
\end{subequations}
Combining \eqref{phi_epsilon} with \eqref{phi_hat_epsilon} yields $\rho^{\rm{opt}}_t = \varphi \widehat{\varphi}$, which recovers \eqref{OptimallyControlledJointPDF}. Substituting $\psi_t = \log{\varphi}$ into \eqref{inputSBP} yields \eqref{OptimalVelocity}.

Substituting $\rho^{\rm{opt}}_t = \varphi \widehat{\varphi}$ into the boundary conditions \eqref{BoundaryConditionsForProp1}, yields new boundary conditions 
\begin{align*}
    \widehat{\varphi}(t=t_0,\cdot)\varphi(t=t_0,\cdot) = \rho_0, \quad \widehat{\varphi}(t=t_1,\cdot)\varphi(t=t_1,\cdot) = \rho_1,
\end{align*}
thus recovering \eqref{factorBC}. 

Additionally, substituting $\psi_t = \log{\varphi}$ into \eqref{HJBSBP} yields
\begin{align*}
    &\dfrac{\partial}{\partial t} \left( \log{\varphi} \right)+\dfrac{1}{2}\vert \nabla \log{\varphi} \vert^{2} + \Delta\log{\varphi} = q(\bm{x}_t^{\bm{u}}),
\end{align*}
which then simplifies to \eqref{FactorPDEBackward}.

Substituting $\rho^{\rm{opt}}_t = \varphi \widehat{\varphi}$ and $\psi_t = \log{\varphi}$ into \eqref{FPKSBP} yields $\dfrac{\partial}{\partial t} \left( \varphi \widehat{\varphi} \right)+ \nabla\cdot \left( \left( \varphi \widehat{\varphi} \right) \nabla\log{\varphi}\right)=\Delta\left( \varphi \widehat{\varphi} \right)$, which expands to
\begin{align*}
    \varphi \dfrac{\partial \widehat{\varphi}}{\partial t} + \widehat{\varphi}\dfrac{\partial \varphi}{\partial t} + \langle \nabla \widehat{\varphi}, \nabla \varphi \rangle + \widehat{\varphi} \Delta \varphi = \varphi \Delta \widehat{\varphi} + 2\langle \nabla \varphi, \nabla \widehat{\varphi} \rangle + \widehat{\varphi} \Delta \varphi. 
\end{align*}
Finally, substituting \eqref{FactorPDEBackward} into the above yields \eqref{FactorPDEForward}.\qedsymbol

\subsection{Proof for Lemma \ref{Lemma:SolToSpatialODE}} \label{proof_of_SolToSpatialODE}
Let $Y(y) = \exp{\left( -\frac{y^2 \sqrt{d}}{2} \right)} f(y)$ be the solution of the parametric ODE \eqref{SpatialODE} for some unknown $f(y)$. 
  Substituting this expression for $Y(y)$ into \eqref{SpatialODE} yields:
\begin{align*}
    &-\sqrt{d} \exp{\left( -\frac{y^2 \sqrt{d}}{2} \right)} f + d y^2 \exp{\left( -\frac{y^2 \sqrt{d}}{2} \right)} f - y\sqrt{d} \exp{\left( -\frac{y^2 \sqrt{d}}{2} \right)} \frac{\differential f}{\differential y} \\ -&y\sqrt{d} \exp{\left( -\frac{y^2 \sqrt{d}}{2} \right)}\frac{\differential f}{\differential y}  + \exp{\left( -\frac{y^2 \sqrt{d}}{2} \right)} \frac{\differential ^2 f}{\differential y^2} - (d y^2 + c)\exp{\left( -\frac{y^2 \sqrt{d}}{2} \right)} f = 0.
\end{align*}
Dividing the above by $\exp{\left( -\frac{y^2 \sqrt{d}}{2} \right)}$ and simplifying, we obtain
\begin{align*}
      \frac{\differential ^2 f}{\differential y^2} - 2 y\sqrt{d} \frac{\differential f}{\differential y} - \left(c + \sqrt{d}\right) f = 0. 
\end{align*}
Introducing a change of variable $y\mapsto \xi := d^{1/4}y$ and letting $F(\xi):=f(d^{-1/4}\xi)$, we rewrite the above ODE as
\begin{align*}  \sqrt{d}\frac{\differential ^2 F}{\differential \xi^2} - 2 d^{-1/4} \xi d^{1/4} \sqrt{d} \frac{\differential F}{\differential \xi} - \left(c + \sqrt{d}\right) F = 0.
\end{align*}
Dividing through by $\sqrt{d}$ gives
\begin{align}
      \frac{\differential ^2 F}{\differential \xi^2} - 2\xi \frac{\differential F}{\differential \xi} - \left(\frac{c}{\sqrt{d}} + 1\right) F = 0,
      \label{hermite_ode}
\end{align}
which is the Hermite's ODE \cite[p. 328-329]{courant2008methods}. 

Assuming $f$, and thus $F$, must be polynomially bounded, the solutions to \eqref{hermite_ode} is $F(\xi)=H_{n}(\xi)$, the Hermite polynomials of order 
$n =  - \frac{c}{2\sqrt{d}} - \frac{1}{2} \in \mathbb{N}_{0}$. Therefore, the solutions to \eqref{SpatialODE} can be written as $Y(y) = a \exp{\left(- \frac{y^2 \sqrt{d}}{2}\right)} H_{n}\left(  d^{1/4} y \right)$
for some constant $a \in \mathbb{R}$ and $n = - \frac{c}{2\sqrt{d}} - \frac{1}{2}\in \mathbb{N}_{0}$.\qedsymbol

\subsection{Proof for Proposition \ref{ParticularSol}}\label{proof_of_ParticularSol}
Substituting $t=t_0$ in \eqref{general_solution} gives
\begin{align}
    \widehat{\eta}_0(\bm{y}) = \sum_{m_n = 0}^{\infty} \ldots \sum_{m_1 = 0}^{\infty} a_{\bm{m}} \prod_{i = 1}^{n} \exp{\left(- t_{0}\sqrt{d_i}\left(2 m_i + 1\right)- \frac{y_i^2 \sqrt{d_i}}{2} \right)} H_{m_i}\!\left( d_i^{1/4}y_i\right).
\label{eta0att0}    
\end{align}
For a fixed $j_{1}\in\mathbb{N}$, multiplying both sides of \eqref{eta0att0}  by $H_{j_1}\left(\!d_1^{1/4}y_1\right)\exp{\left(-\frac{y_1^2 \sqrt{d_1}}{2}\right)}$,
and integrating over $y_1$, we get
\begin{align*}
    &\int_{-\infty}^{\infty} H_{j_1}\left(d_1^{1/4} y_1\right)\exp{\left(- \frac{y_1^2 \sqrt{d_1}}{2} \right)} \widehat{\eta}_0(\bm{y})\differential y_1 \\ 
    &= \int_{-\infty}^{\infty} \sum_{m_n = 0}^{\infty} \ldots \sum_{m_1 = 0}^{\infty} a_{\bm{m}} \prod_{i = 2}^{n} \exp{\left(- t_{0}\sqrt{d_i}\left(2 m_i + 1\right)- \frac{y_i^2 \sqrt{d_i}}{2} \right)} \\
    & \hspace{20pt}H_{m_i}\left( d_i^{1/4}y_i \right) H_{m_1}\left(d_1^{1/4} y_1\right) H_{j_1}\left(d_1^{1/4} y_1\right)\exp{\left(- t_{0}\sqrt{d_1}\left(2 m_1 + 1\right)- y_1^2 \sqrt{d_1} \right)} \differential y_1.
\end{align*}
Using \eqref{OrthogonalityHermitePoly}, we simplify the above as follows:
\begin{align*}
    &\int_{-\infty}^{\infty} H_{j_1}(d_1^{1/4} y_1)\exp{\left(- \frac{y_1^2 \sqrt{d_1}}{2} \right)} \widehat{\eta}_0(\bm{y}) \differential y_1 \\ 
    &= \sum_{m_n = 0}^{\infty} \ldots \sum_{m_2 = 0}^{\infty} a_{(j_1, m_2, \ldots, m_n)} \prod_{i = 2}^{n} \exp{\left(-t_{0}\sqrt{d_{i}}\left(2m_{i}+1\right)- \frac{y_i^2 \sqrt{d_i}}{2} \right)} H_{m_i}\left( d_i^{1/4}y_i \right)  \\
    &\hspace{50pt}\int_{-\infty}^{\infty}  H_{j_1}(d_1^{1/4} y_1) H_{j_1}(d_1^{1/4} y_1)\exp{\left(-t_{0}\sqrt{d_{1}}\left(2j_{1}+1\right)- y_1^2 \sqrt{d_1} \right)} dy_1\\
    &= \sum_{m_n = 0}^{\infty} \ldots \sum_{m_2 = 0}^{\infty} a_{(j_1, m_2, \ldots, m_n)} \prod_{i = 2}^{n} \exp{\left(-t_{0}\sqrt{d_{i}}\left(2m_{i}+1\right)- \frac{y_i^2 \sqrt{d_i}}{2} \right)} H_{m_i}\left( d_i^{1/4}y_i \right)\\
&\hspace{130pt}d_1^{-1/4} \sqrt{\pi}\: 2^{j_1}\:\left(j_{1}!\right)\exp{\left(-t_{0}\sqrt{d_{1}}\left(2j_{1}+1\right)\right)}.
\end{align*}
Repeating this process $n-1$ more times, yields
\begin{align*}
    &\int_{-\infty}^{\infty} \ldots \int_{-\infty}^{\infty} H_{j_1}\left(d_1^{1/4} y_1\right)\exp{\left(- \frac{y_1^2 \sqrt{d_1}}{2} \right)} \ldots  \\ 
    & \hspace{130pt}H_{j_n}\left(d_n^{1/4} y_n\right)\exp{\left(- \frac{y_n^2 \sqrt{d_n}}{2} \right)}\: \widehat{\eta}_0(\bm{y}) \:\differential y_1 \ldots \differential y_n\\
    & = a_{(j_1, j_2, \ldots, j_n)} \prod_{i = 1}^{n} \left( d_i^{-1/4} \sqrt{\pi} 2^{j_i}\left(j_{i}!\right) \exp{\left(-t_{0}\sqrt{d_{i}}\left(2j_{i}+1\right)\right)}\right).
\end{align*}
Therefore, for any index vector $\bm{j}=(j_1,j_2,\ldots,j_n)\in\mathbb{N}_{0}^{n}$, we have
\begin{align}
    a_{\bm{j}} &= \int_{-\infty}^{\infty} \ldots \int_{-\infty}^{\infty} \prod_{i = 1}^{n} \left( \frac{d_i^{1/4}} {\sqrt{\pi} 2^{j_i}\left(j_{i}!\right)} H_{j_i}\left(d_i^{1/4} y_i\right)\exp{\left(t_{0}\sqrt{d_{i}}\left(2j_{i}+1\right)- \frac{y_i^2 \sqrt{d_i}}{2} \right)} \right) \nonumber\\
&\hspace{250pt}\widehat{\eta}_0\left(\bm{y}\right)\differential y_1 \ldots \differential y_n.
\label{GeneralCoeff}
\end{align}
Substituting \eqref{GeneralCoeff} for the coefficients in the general solution \eqref{general_solution}, and using $\bm{z}$ as the dummy integration variable, we arrive at \eqref{general_solution_with_coeff}-\eqref{integrand}.\qedsymbol

\subsection{Proof for Theorem \ref{Thm:Kernel}} \label{proof_of_Thm:Kernel}
We start by re-writing \eqref{general_solution_with_coeff}-\eqref{integrand} as
\begin{align}
    \widehat{\eta}(t, \bm{y})= &\int_{-\infty}^{\infty} \hdots\int_{-\infty}^{\infty} \prod_{i = 1}^{n} \bigg\{ \frac{d_i^{1/4}}{\sqrt{\pi}} \exp{\left( - \frac{(y_i^2 + z_i^2)\sqrt{d_i}}{2} - \left(t-t_{0}\right) \sqrt{d_i} \right)}(\heartsuit)\bigg\} \label{EtaHatSemifinal}\\
&\hspace*{210pt}\widehat{\eta}_0(\bm{z})\:\differential z_1 \ldots \differential z_n\nonumber, 
\end{align}
where
\begin{align}
    (\heartsuit) &:= \sum_{m_i = 0}^{\infty}\!\!\left( \frac{1}{2^{m_i}\left(m_i!\right)} \exp{(-2m_i\left(t-t_{0}\right)\sqrt{d_i})} H_{m_i}\left(d_i^{1/4}y_i\right)H_{m_i}\left(d_i^{1/4}z_i\right)\!\!\right).
\label{defHeart}    
\end{align}
For fixed $c>0$, we use the Hermite polynomial identity:
\begin{align*}
    H_{r}(cx) = \frac{1}{c^r} e^{(cx)^2} \left( - \frac{\differential}{\differential x} \right)^r \left( e^{-(cx)^2}\right) = (-2\complexi)^r \frac{1}{\sqrt{\pi}} \int_{0}^{\infty} e^{-(\theta-cx\complexi)^2} \theta^r \differential \theta, \quad r\in\mathbb{N}_{0},
\end{align*}
where the first equality is due to definition \eqref{defPhysicistHermitPoly}, and the second equality is due to the Fourier representation 
\begin{align}
e^{-x^2} = \int e^{-\theta^2 + 2\complexi x \theta}\differential \theta/\sqrt{\pi}.
\label{FourierRepresentation}    
\end{align}
Therefore,
\begin{align}
&(\heartsuit) 
    = \sum_{m_i = 0}^{\infty} \frac{1}{2^{m_i}\left(m_i!\right)} \exp{\left(-2m_{i}\left(t-t_{0}\right)\sqrt{d_i}\right)} \frac{1}{d_i^{m_i/4}} e^{\left(d_i^{1/4}y_i\right)^2} \left( - \!\!\frac{\differential}{\differential y_i} \right)^{m_{i}} \!\!e^{-\left(d_i^{1/4}y_i\right)^2} \nonumber\\ & \qquad\qquad\qquad\qquad\times (-2\complexi)^{m_i} \frac{1}{\sqrt{\pi}} \int_{t_{0}}^{\infty} e^{-\left(\theta-t_{0}-d_i^{1/4}z_i\complexi\right)^2} \left(\theta-t_{0}\right)^{m_{i}} \differential \theta  \nonumber\\
    &= \hspace{-2pt}\frac{e^{\left(d_i^{1/4}y_i\right)^2} }{\sqrt{\pi}} \int_{t_{0}}^{\infty} \sum_{m_i = 0}^{\infty} \left(\frac{(\varrho_{i}  \complexi \left(\theta-t_{0}\right) )^{m_i}}{\left(m_i!\right)\left(d_{i}^{1/4}\right)^{m_i}}\left( \frac{\differential}{\differential y_i} \right)^{m_{i}} \!\!e^{-\left(d_i^{1/4}y_i\right)^2} \!\right) e^{-\left(\theta-t_{0}-d_i^{1/4}z_i\complexi\right)^2} \differential \theta \nonumber
\end{align}
where $\varrho_i := \exp{(-2\left(t-t_{0}\right)\sqrt{d_i})}$. Invoking Lemma \ref{LemmaExponentialOfQuadAsSeries}, we note that
\begin{align*}
    \sum_{m_i = 0}^{\infty} \left(\frac{(\varrho_{i}  \complexi \left(\theta-t_{0}\right) )^{m_i}}{\left(m_i!\right)\left(d_{i}^{1/4}\right)^{m_i}} \left( \frac{\differential}{\differential y_i} \right)^{m_{i}} \hspace{-5pt}\left( e^{-(d_i^{1/4}y_i)^2} \right) \hspace{-4pt}\right) = e^{-\left(d_i^{1/4}y_i + \complexi\varrho_i\left(\theta-t_{0}\right)\right)^{2}},
\end{align*}
and so, $(\heartsuit)$ simplifies to
\begin{align}
\dfrac{e^{d_{i}^{1/2}z_{i}^{2}}}{\sqrt{\pi}}\int_{t_0}^{\infty}\exp\left(-(1-\varrho_i^2)(\theta-t_0)^2 + 2d_i^{1/4}(\theta-t_0)\complexi\left(z_{i} - y_i\varrho_i\right)\right)\differential\theta
\label{SimplifiedIntegral}.
\end{align}
Noting that $0<\varrho_i<1$, we evaluate \eqref{SimplifiedIntegral} by substituting $\sigma := \sqrt{1 - \varrho_i^2}\left(\theta-t_{0}\right)$ as
\begin{align}
(\heartsuit)=&\dfrac{e^{d_{i}^{1/2}z_{i}^{2}}}{\sqrt{1-\varrho_i^2}\sqrt{\pi}}\int_{0}^{\infty}\exp\left(-\sigma^2 + 2\complexi\sigma\dfrac{d_i^{1/4}\left(z_i - y_i\varrho_i\right)}{\sqrt{1-\varrho_i^2}}\right)\differential \sigma\nonumber\\
=& \dfrac{1}{\sqrt{1-\varrho_i^2}}\exp\left(\dfrac{2\varrho_i d_{i}^{1/2} y_i z_i  - \varrho_i^2 d_i^{1/2}\left(y_i^2 + z_i^2\right)}{1-\varrho_i^2}\right),
\label{ReallySimplifiedIntegral}
\end{align}
where the last equality follows from \eqref{FourierRepresentation}. Combining \eqref{ReallySimplifiedIntegral} with \eqref{EtaHatSemifinal}, we obtain $
    \widehat{\eta}(t, \bm{y} ) = \int_{-\infty}^{\infty} \ldots \int_{-\infty}^{\infty} \kappa_{++}(t_0, \bm{y}, t, \bm{z}) \widehat{\eta}_0(\bm{z}) \differential z_1 \ldots \differential z_n$ where the kernel
\begin{multline*}
    \kappa_{++}(t_0, \bm{y}, t, \bm{z}) := \prod_{i = 1}^{n} \bigg\{ \frac{d_i^{1/4}}{\sqrt{\pi}} \exp{\left( - \frac{(y_i^2 + z_i^2)\sqrt{d_i}}{2} - \left(t-t_{0}\right) \sqrt{d_i} \right)} \\
    \frac{1}{\sqrt{1 - \varrho_i^2}} \exp\left(\dfrac{2\varrho_i d_{i}^{1/2} y_i z_i  - \varrho_i^2 d_i^{1/2}\left(y_i^2 + z_i^2\right)}{1-\varrho_i^2}\right) \bigg\}.
\end{multline*}
Next, we rearrange the above expression for $\kappa_{++}$ as
\begin{align}
    \kappa_{++}(t_{0}, \bm{y}, t, \bm{z}) = \frac{1}{\pi^{n/2}} \left(\prod_{i = 1}^{n}  d_i^{1/4} \right) \left(\prod_{i = 1}^{n}  \frac{1}{\sqrt{1 - \varrho_i^2}} \right) 
    \left(\prod_{i = 1}^{n} \exp{\left( - \left(t-t_{0}\right) \sqrt{d_i} \right)} \right) \times \label{RearrangedKernel}\\ \left(\prod_{i = 1}^{n}  \exp{\left( \left( - \frac{1}{2} - \frac{\varrho_i^2}{1 - \varrho_i^2} \right)(y_i^2 + z_i^2)\sqrt{d_i} \right)} \right) \left(\prod_{i = 1}^{n}  \exp{\left( \frac{2 \varrho_i}{1-\varrho_i^2} y_i z_i \sqrt{d_i} \right)} \right).\nonumber
\end{align}
Recalling the definitions of hyperbolic functions and that $\varrho_i := \exp{(-2\left(t-t_{0}\right)\sqrt{d_i})}$, we find: $1/\sqrt{1 - \varrho_i^2} = e^{\left(t-t_{0}\right)\sqrt{d_i}}/\sqrt{2\sinh(2\left(t-t_{0}\right)\sqrt{d_i})}$, $2 \varrho_i/(1 - \varrho_i^2) = 1/\sinh(2\left(t-t_{0}\right)\allowbreak\sqrt{d_i})$, and $-1/2 - \varrho_i^2/(1 - \varrho_i^2) = -\coth(2\left(t-t_{0}\right)\sqrt{d_i})/2$. Hence, \eqref{RearrangedKernel} can be expressed as
\begin{align}
    \kappa_{++}(t_0, \bm{y}, t, \bm{z}) = \frac{\left(\prod_{i = 1}^{n}  d_i^{1/4} \right)}{(2\pi)^{n/2}}  \left(\prod_{i = 1}^{n}  \frac{1}{\sqrt{\sinh(2\left(t-t_{0}\right)\sqrt{d_i})}} \right) \times \label{KernelFinal}\\
    \left(\prod_{i = 1}^{n}  \exp{\left( \frac{\sqrt{d_i}}{\sinh(2\left(t-t_{0}\right)\sqrt{d_i})} \left( y_i z_i + \left(- \frac{1}{2} \cosh(2\left(t-t_{0}\right)\sqrt{d_i}) \right)(y_i^2 + z_i^2) \right) \right) } \right).\nonumber
\end{align}
Letting ${\black{\bm{M}_{tt_{0}}}}$ as in \eqref{Matrix_M_Def}
and noting that $\det{\left({\black{\bm{M}_{tt_{0}}}} \right)} = \prod_{i=1}^n d_i>0$, \eqref{KernelFinal} simplifies to \eqref{kernel_function}.\qedsymbol

\subsection{Proof for Proposition \ref{Proposition:MatrixM}}\label{ProofMatrixM}
(i) Follows from direct computation.\\
\noindent(ii) Thanks to the hyperbolic function identity $\coth^{2}(\cdot)-{\rm{csch}}^{2}(\cdot)=1$, definition \eqref{defSymplecticGroup} is satisfied, i.e.,  $\left({\black{\bm{M}_{tt_{0}}^{(1)}\bm{M}_{tt_{0}}^{(2)}}}\right)^{\top}\bm{J}\allowbreak\left({\black{\bm{M}_{tt_{0}}^{(1)}\bm{M}_{tt_{0}}^{(2)}}}\right)=\bm{J}$. Thus ${\black{\bm{M}_{tt_{0}}^{(1)}\bm{M}_{tt_{0}}^{(2)}}}\in{\rm{Sp}}\left(2n,\mathbb{R}\right)$.

To see positive definiteness, first note that ${\black{\bm{M}_{tt_{0}}^{(1)}}}\succ 0$ for $\bm{D}\succ 0$ since
\begin{align*}
\begin{pmatrix} 
\bm{y} & \bm{z}
\end{pmatrix}{\black{\bm{M}_{tt_{0}}^{(1)}}}\begin{pmatrix}
\bm{y}\\ 
\bm{z}
\end{pmatrix} &= \begin{pmatrix} 
\bm{y} & \bm{z}
\end{pmatrix}\begin{bmatrix}
\cosh\left(2{\black{\left(t-t_{0}\right)}}\sqrt{\bm{D}}\right) & -\bm{I}\\
-\bm{I} & \cosh\left(2{\black{\left(t-t_{0}\right)}}\sqrt{\bm{D}}\right)
\end{bmatrix}\begin{pmatrix} 
\bm{y}\\ 
\bm{z}
\end{pmatrix}\\
&= \bm{y}^{\top}\underbrace{\cosh\left(2{\black{\left(t-t_{0}\right)}}\sqrt{\bm{D}}\right)}_{\succ \bm{I}}\bm{y} -2\langle\bm{y},\bm{z}\rangle + \bm{z}^{\top}\underbrace{\cosh\left(2{\black{\left(t-t_{0}\right)}}\sqrt{\bm{D}}\right)}_{\succ \bm{I}}\bm{z}\\
&>\|\bm{y}-\bm{z}\|_2^2 \geq 0,
\end{align*}
where the second equality follows because $\cosh(\cdot)>1$ for nonzero argument.

Also, $\bm{D}\succ \bm{0}$ implies the positive diagonal matrix ${\black{\bm{M}_{tt_{0}}^{(2)}}}\succ \bm{0}$, and therefore, 
\begin{align*}
{\rm{eig}}\left({\black{\bm{M}_{tt_{0}}^{(1)}\bm{M}_{tt_{0}}^{(2)}}}\right) &= {\rm{eig}}\left({\black{\bm{M}_{tt_{0}}^{(1)}\left(\bm{M}_{tt_{0}}^{(2)}\right)^{1/2}\left(\bm{M}_{tt_{0}}^{(2)}\right)^{1/2}}}\right)\\&= {\rm{eig}}\left(\underbrace{{\black{\left(\bm{M}_{tt_{0}}^{(2)}\right)^{1/2}\bm{M}_{tt_{0}}^{(1)}\left(\bm{M}_{tt_{0}}^{(2)}\right)^{1/2}}}}_{\text{congruence transform of ${\black{\bm{M}_{tt_{0}}^{(1)}}}$}}\right)>0.   
\end{align*}
Additionally, ${\black{\bm{M}_{tt_{0}}^{(1)}\bm{M}_{tt_{0}}^{(2)}}}$ being symmetric, we have that ${\black{\bm{M}_{tt_{0}}^{(1)}\bm{M}_{tt_{0}}^{(2)}}}\succ \bm{0}$.\\
\noindent(iii) That
$${\rm{eig}}\left({\black{\bm{M}_{tt_{0}}}}\right) = {\rm{eig}}\left(\underbrace{\begin{bmatrix}
\bm{D}^{1/4} & \bm{0}\\
\bm{0} & \bm{D}^{1/4}
\end{bmatrix}{\black{\bm{M}_{tt_{0}}^{(1)}\bm{M}_{tt_{0}}^{(2)}}}\begin{bmatrix}
\bm{D}^{1/4} & \bm{0}\\
\bm{0} & \bm{D}^{1/4}
\end{bmatrix}}_{\text{congruence transform of ${\black{\bm{M}_{tt_{0}}^{(1)}\bm{M}_{tt_{0}}^{(2)}}}$}}\right) > 0,$$
and ${\black{\bm{M}_{tt_{0}}}}$ is symmetric, together implies ${\black{\bm{M}_{tt_{0}}}}\succ\bm{0}$.\qedsymbol

{\black{\subsection{Proof for Proposition \ref{prop:Geodesic}}\label{proof_of_prop_geodesic}
That the minimum in \eqref{ActionIntegral} is attained for the $L$ in \eqref{defLagrangianGeodesic}, follows from the strict convexity and coercivity of $\bm{v}\mapsto L\left(\cdot,\bm{v}\right)$.

For notational convenience, let $\omega_{i}:=2\sqrt{d_{i}}$ $\forall i=1,\hdots,n$. The optimal solution $\bm{q}_{\rm{opt}}(\tau)$, $\tau\in[t_0,t]$, solves the Euler-Lagrange equation $\frac{\differential^2}{\differential\tau^2}{\bm{q}}^{\rm{opt}} = 4\bm{Dq}^{\rm{opt}}$, or in component form: $\frac{\differential^2}{\differential\tau^2}{q_i}^{\rm{opt}} = \omega_{i}^{2}q_{i}^{\rm{opt}}$ subject to the endpoint constraints $q_{i}^{\rm{opt}}(t_0)=y_{i}$, $q_{i}^{\rm{opt}}(t)=z_{i}$. The general solution is $q_{i}^{\rm{opt}}(\tau) = a_{i}e^{\omega_{i}\tau} + b_{i}e^{-\omega_{i}\tau}$, and the endpoint constraints determine the constants $\left(a_i,b_i\right)$ as  
\begin{align}
\left(a_i,b_i\right)=\left(-y_{i}e^{-\omega_{i}t} + z_{i}e^{-\omega_{i}t_{0}},\:y_{i}e^{\omega_{i}t} - z_{i}e^{\omega_{i}t_{0}}\right)/2\sinh\left(\omega_{i}\left(t-t_{0}\right)\right).
\label{Constantsaibi}
\end{align}
Then, \eqref{ActionIntegral} gives
\begin{align}
\dist_{tt_{0}}^2(\bm{y},\bm{z}) &= \int_{t_{0}}^{t} \sum_{i=1}^{n}\left(\dfrac{1}{2}\left(\dfrac{\differential}{\differential\tau}q_{i}^{\rm{opt}}(\tau)\right)^2 + \dfrac{1}{2}\omega_{i}^{2}\left(q_{i}^{\rm{opt}}(\tau)\right)^{2}\right)\differential\tau \nonumber\\
&= \displaystyle\sum_{i=1}^{n}\omega_{i}^{2}\int_{t_{0}}^{t} \left(a_{i}^{2}e^{2\omega_{i}\tau} + b_{i}^{2}e^{-2\omega_{i}\tau}\right)\differential\tau\nonumber\\
&=\sum_{i=1}^{n}\dfrac{\omega_{i}}{2}\big\{a_{i}^{2}\left(e^{2\omega_{i}t}-e^{2\omega_{i}t_{0}}\right) - b_{i}^{2}\left(e^{-2\omega_{i}t}-e^{-2\omega_{i}t_{0}}\right)\big\}\nonumber\\
& = \sum_{i=1}^{n}\dfrac{\omega_{i}}{2}\dfrac{\left(y_{i}^{2}+z_{i}^{2}\right)\cosh\left(\omega_{i}\left(t-t_{0}\right)\right) - 2y_{i}z_{i}}{\sinh\left(\omega_{i}\left(t-t_{0}\right)\right)},\label{distsquaredFinal}
\end{align}
where in the last line, we substituted for $(a_i,b_i)$ from \eqref{Constantsaibi}, and simplified the resulting expression using the identity $\sinh\left(2\left(\cdot\right)\right) = 2\sinh\left(\cdot\right)\cosh\left(\cdot\right)$.

Since $\omega_{i}=2\sqrt{d_{i}}$ $\forall i=1,\hdots,n$, we see that \eqref{distsquaredFinal} is indeed the desired expression, c.f. the expression inside the exponential in \eqref{KernelFinal}. This completes the proof.
\qedsymbol
}}

\subsection{Proof for Theorem \ref{ThmHeatKernelSplCase}}\label{proof_of_heat_kernel_spl_case}
Notice that the kernel \eqref{kernel_function} is a separable product in tuple $(t-t_0, d_i, y_i, z_i)$, cf. \eqref{KernelFinal}. Letting $\delta_i := \sqrt{d_i}$ $\forall i=1,\hdots, n$, we then have 
\begin{align}
&\displaystyle\lim_{(d_1,\hdots,d_n)\downarrow\bm{0}}\kappa_{++}(t_0, \bm{y}, t, \bm{z}) = \dfrac{1}{(2\pi)^{n/2}}\prod_{i=1}^{n}\bigg\{\displaystyle\lim_{\delta_{i}\downarrow 0}\sqrt{\dfrac{\delta_i}{\sinh(2\left(t-t_{0}\right)\delta_i)}}\nonumber\\
&\exp\left(-\frac{1}{2}\delta_i\left(\coth(2(t-t_0)\delta_i)(y_i^2 + z_i^2) - 2{\rm{csch}}(2(t-t_0)\delta_i)y_i z_i\right)\right)\bigg\}\nonumber\\
&=\dfrac{1}{(2\pi)^{n/2}}\prod_{i=1}^{n}\sqrt{\ell_{1}}\exp\left(-\frac{1}{2}\begin{pmatrix} 
y_i & z_i
\end{pmatrix} \begin{bmatrix}
\ell_{2} & -\ell_{3}\\
-\ell_{3} & \ell_{2} 
\end{bmatrix}\begin{pmatrix} 
y_i\\ 
z_i
\end{pmatrix}
\right),
\label{LimitForm}
\end{align}
where $\ell_{1}:=\displaystyle\lim_{\delta_{i}\downarrow 0}\dfrac{\delta_i}{\sinh(2\left(t-t_{0}\right)\delta_i)}$, $\ell_{2}:= \displaystyle\lim_{\delta_i\downarrow 0}\delta_i\coth(2(t-t_0)\delta_i)$, and $\ell_{3}:= \displaystyle\lim_{\delta_i\downarrow 0}\delta_i{\rm{csch}}\allowbreak(2(t-t_0)\delta_i)$. 

We evaluate the limits $\ell_1,\ell_2,\ell_3$ by applying L'H\^{o}pital's rule: $\ell_{1}=\ell_{2}=\ell_{3}=1/(2(t-t_0))$. Substituting these back in \eqref{LimitForm} yields $\displaystyle\lim_{(d_1,\hdots,d_n)\downarrow\bm{0}}\kappa_{++}(t_0, \bm{y}, t, \bm{z}) =\left(4\pi(t-t_{0})\right)^{-n/2}\exp\left(-\frac{\vert\bm{y}-\bm{z}\vert^2}{4(t-t_{0})}\right)=\kappa_{0}(t_0, \bm{y}, t, \boldsymbol{z})$,
as claimed.\qedsymbol

\subsection{Proof for Theorem \ref{SemiPosDefKernel}} \label{proof_of_SemiPosDefKernel}
We write 
\begin{align*}
&\kappa_{+}\left(t_0,\bm{y},t,\bm{z}\right) = \displaystyle\lim_{\left(d_{i_{n-p+1}},\hdots,d_{i_{n}}\right)\downarrow\bm{0}}\kappa_{++}\left(t_0,\bm{y},t,\bm{z}\right)\\
&= \kappa_{++}\left(t_0,\bm{y}_{\left[i_{1}:i_{n-p}\right]},t,\bm{z}_{\left[i_{1}:i_{n-p}\right]}\right)\displaystyle\lim_{\left(d_{i_{n-p+1}},\hdots,d_{i_{n}}\right)\downarrow\bm{0}}\kappa_{++}\left(t_0,\bm{y}_{\left[i_{n-p+1}:i_{n}\right]},t,\bm{z}_{\left[i_{n-p+1}:i_{n}\right]}\right),
\end{align*}
where the last equality is due to the separable (in spatial coordinates) product structure in $\kappa_{++}$. Invoking Theorem \ref{ThmHeatKernelSplCase} for the remaining limit, the result follows.\qedsymbol

\bibliographystyle{siamplain}
\bibliography{references}
\end{document}

%% file: ex_shared.tex
\usepackage{lipsum}
\usepackage{amsfonts}
\usepackage{graphicx}
\usepackage{epstopdf}
\usepackage{algorithmic}
\usepackage[utf8]{inputenc} 
\usepackage[T1]{fontenc}    
\usepackage{hyperref}       
\usepackage{url}            
\usepackage{booktabs}       
\usepackage{amsfonts,amsmath,amssymb,bm,cite,graphicx,subcaption,wrapfig}       
\usepackage{nicefrac}       
\usepackage{microtype}
\usepackage{xcolor} 
\usepackage{bm}
\ifpdf
  \DeclareGraphicsExtensions{.eps,.pdf,.png,.jpg}
\else
  \DeclareGraphicsExtensions{.eps}
\fi

\newcommand{\red}{\color{red}}

\newcommand{\black}{\color{black}}
\newcommand{\differential}{{\rm{d}}}
\newcommand{\dist}{{\rm{dist}}}
\newcommand{\complexi}{{\mathrm{i}}}

\renewcommand{\qedsymbol}{\hfill\ensuremath{\blacksquare}}

\newsiamremark{remark}{Remark}
\newsiamremark{hypothesis}{Hypothesis}
\crefname{hypothesis}{Hypothesis}{Hypotheses}
\newsiamthm{claim}{Claim}

\headers{Schr\"{o}dinger Bridge with Quadratic State Cost is Exactly Solvable}{Alexis M.H. Teter, Wenqing Wang, and Abhishek Halder}

\title{Schr\"{o}dinger Bridge with Quadratic State Cost is Exactly Solvable\thanks{Submitted to the editors DATE.
\funding{This work was supported by NSF grants 2112755 and 2111688.}}}

\author{Alexis M.H. Teter\thanks{Department of Applied Mathematics, University of California Santa Cruz (\email{amteter@ucsc.edu}).}
\and Wenqing Wang\thanks{Department of Aerospace Engineering, Iowa State University (\email{wqwang@iastate.edu}).}
\and Abhishek Halder\thanks{Department of Aerospace Engineering, Iowa State University; Department of Applied Mathematics, University of California Santa Cruz (\email{ahalder@iastate.edu})}}

\usepackage{amsopn}


%% file: SIAM-SBP-quadratic-state-cost.bbl
\begin{thebibliography}{10}

\bibitem{AbraSteg72}
{\sc M.~Abramowitz and I.~A. Stegun}, eds., {\em Handbook of Mathematical Functions with Formulas, Graphs, and Mathematical Tables}, U.S. Government Printing Office, Washington, DC, USA, tenth printing~ed., 1972.

\bibitem{aebi1996schrodinger}
{\sc R.~Aebi}, {\em Schr{\"o}dinger diffusion processes}, Springer Science \& Business Media, 1996.

\bibitem{aebi1992large}
{\sc R.~Aebi and M.~Nagasawa}, {\em Large deviations and the propagation of chaos for {S}chr{\"o}dinger processes}, Probability Theory and Related Fields, 94 (1992), pp.~53--68.

\bibitem{anderson2007optimal}
{\sc B.~D. Anderson and J.~B. Moore}, {\em Optimal control: linear quadratic methods}, Courier Corporation, 2007.

\bibitem{benamou2000computational}
{\sc J.-D. Benamou and Y.~Brenier}, {\em A computational fluid mechanics solution to the {M}onge-{K}antorovich mass transfer problem}, Numerische Mathematik, 84 (2000), pp.~375--393, \url{https://doi.org/10.1007/s002119900117}.

\bibitem{beurling1960automorphism}
{\sc A.~Beurling}, {\em An automorphism of product measures}, Annals of Mathematics,  (1960), pp.~189--200.

\bibitem{born1926quantenmechanik}
{\sc M.~Born}, {\em Quantenmechanik der sto{\ss}vorg{\"a}nge}, Zeitschrift f{\"u}r physik, 38 (1926), pp.~803--827.

\bibitem{borwein2013closed}
{\sc J.~M. Borwein, R.~E. Crandall, et~al.}, {\em Closed forms: what they are and why we care}, Notices of the AMS, 60 (2013), pp.~50--65.

\bibitem{bunne2023schrodinger}
{\sc C.~Bunne, Y.-P. Hsieh, M.~Cuturi, and A.~Krause}, {\em The {S}chr{\"o}dinger bridge between {G}aussian measures has a closed form}, in International Conference on Artificial Intelligence and Statistics, PMLR, 2023, pp.~5802--5833.

\bibitem{bushell1973hilbert}
{\sc P.~J. Bushell}, {\em Hilbert's metric and positive contraction mappings in a {B}anach space}, Archive for Rational Mechanics and Analysis, 52 (1973), pp.~330--338.

\bibitem{caluya2021reflected}
{\sc K.~F. Caluya and A.~Halder}, {\em Reflected {S}chr{\"o}dinger bridge: Density control with path constraints}, in 2021 American Control Conference (ACC), IEEE, 2021, pp.~1137--1142.

\bibitem{caluya2021wasserstein}
{\sc K.~F. Caluya and A.~Halder}, {\em Wasserstein proximal algorithms for the {S}chr{\"o}dinger bridge problem: density control with nonlinear drift}, IEEE Transactions on Automatic Control, 67 (2021), pp.~1163--1178.

\bibitem{chen2021likelihood}
{\sc T.~Chen, G.-H. Liu, and E.~Theodorou}, {\em Likelihood training of {S}chr{\"o}dinger bridge using forward-backward {SDE}s theory}, in International Conference on Learning Representations, 2021.

\bibitem{chen2015optimalACC}
{\sc Y.~Chen, T.~Georgiou, and M.~Pavon}, {\em Optimal steering of inertial particles diffusing anisotropically with losses}, in 2015 American Control Conference (ACC), IEEE, 2015, pp.~1252--1257.

\bibitem{chen2016entropic}
{\sc Y.~Chen, T.~Georgiou, and M.~Pavon}, {\em Entropic and displacement interpolation: a computational approach using the {H}ilbert metric}, SIAM Journal on Applied Mathematics, 76 (2016), pp.~2375--2396, \url{https://doi.org/10.1137/16M1061382}.

\bibitem{chen2015optimal}
{\sc Y.~Chen, T.~T. Georgiou, and M.~Pavon}, {\em Optimal steering of a linear stochastic system to a final probability distribution, {Part I}}, IEEE Transactions on Automatic Control, 61 (2015), pp.~1158--1169.

\bibitem{chen2016relation}
{\sc Y.~Chen, T.~T. Georgiou, and M.~Pavon}, {\em On the relation between optimal transport and {S}chr{\"o}dinger bridges: A stochastic control viewpoint}, Journal of Optimization Theory and Applications, 169 (2016), pp.~671--691.

\bibitem{8249875}
{\sc Y.~Chen, T.~T. Georgiou, and M.~Pavon}, {\em Optimal steering of a linear stochastic system to a final probability distribution—part {III}}, IEEE Transactions on Automatic Control, 63 (2018), pp.~3112--3118, \url{https://doi.org/10.1109/TAC.2018.2791362}.

\bibitem{chen2021stochastic}
{\sc Y.~Chen, T.~T. Georgiou, and M.~Pavon}, {\em Stochastic control liaisons: {R}ichard {S}inkhorn meets {G}aspard {M}onge on a {S}chr{\"o}dinger bridge}, Siam Review, 63 (2021), pp.~249--313.

\bibitem{cole1951quasi}
{\sc J.~D. Cole}, {\em On a quasi-linear parabolic equation occurring in aerodynamics}, Quarterly of applied mathematics, 9 (1951), pp.~225--236.

\bibitem{courant2008methods}
{\sc R.~Courant and D.~Hilbert}, {\em Methods of Mathematical Physics, Volume 1}, vol.~1, John Wiley \& Sons, 2008.

\bibitem{csiszar1984sanov}
{\sc I.~Csisz{\'a}r}, {\em Sanov property, generalized {I}-projection and a conditional limit theorem}, The Annals of Probability,  (1984), pp.~768--793.

\bibitem{cuturi2013sinkhorn}
{\sc M.~Cuturi}, {\em Sinkhorn distances: Lightspeed computation of optimal transport}, Advances in neural information processing systems, 26 (2013).

\bibitem{dawson1990schrodinger}
{\sc D.~Dawson, L.~Gorostiza, and A.~Wakolbinger}, {\em Schr{\"o}dinger processes and large deviations}, Journal of mathematical physics, 31 (1990), pp.~2385--2388, \url{https://doi.org/10.1063/1.528840}.

\bibitem{de2021diffusion}
{\sc V.~De~Bortoli, J.~Thornton, J.~Heng, and A.~Doucet}, {\em Diffusion {S}chr{\"o}dinger bridge with applications to score-based generative modeling}, Advances in Neural Information Processing Systems, 34 (2021), pp.~17695--17709.

\bibitem{dembo2009large}
{\sc A.~Dembo and O.~Zeitouni}, {\em Large deviations techniques and applications}, vol.~38, Springer Science \& Business Media, 2009.

\bibitem{deng2024reflected}
{\sc W.~Deng, Y.~Chen, N.~T. Yang, H.~Du, Q.~Feng, and R.~T. Chen}, {\em Reflected {S}chr{\"o}dinger bridge for constrained generative modeling}, arXiv preprint arXiv:2401.03228,  (2024).

\bibitem{follmer1988random}
{\sc H.~F{\"o}llmer}, {\em Random fields and diffusion processes}, Lect. Notes Math, 1362 (1988), pp.~101--204.

\bibitem{fortet1940resolution}
{\sc R.~Fortet}, {\em R{\'e}solution d'un syst{\`e}me d'{\'e}quations de {M}. {S}chr{\"o}dinger}, Journal de Math{\'e}matiques Pures et Appliqu{\'e}es, 19 (1940), pp.~83--105.

\bibitem{gantmakher2000theory}
{\sc F.~R. Gantmakher}, {\em The theory of matrices, Vol. 1}, American Mathematical Soc., 2000.

\bibitem{gentil2017analogy}
{\sc I.~Gentil, C.~L{\'e}onard, and L.~Ripani}, {\em About the analogy between optimal transport and minimal entropy}, in Annales de la Facult{\'e} des sciences de Toulouse: Math{\'e}matiques, vol.~26, 2017, pp.~569--600.

\bibitem{gushchin2024entropic}
{\sc N.~Gushchin, A.~Kolesov, A.~Korotin, D.~P. Vetrov, and E.~Burnaev}, {\em Entropic neural optimal transport via diffusion processes}, Advances in Neural Information Processing Systems, 36 (2024).

\bibitem{gushchin2024building}
{\sc N.~Gushchin, A.~Kolesov, P.~Mokrov, P.~Karpikova, A.~Spiridonov, E.~Burnaev, and A.~Korotin}, {\em Building the bridge of {S}chr{\"o}dinger: A continuous entropic optimal transport benchmark}, Advances in Neural Information Processing Systems, 36 (2024).

\bibitem{halder2020hopfield}
{\sc A.~Halder, K.~F. Caluya, B.~Travacca, and S.~J. Moura}, {\em Hopfield neural network flow: A geometric viewpoint}, IEEE Transactions on Neural Networks and Learning Systems, 31 (2020), pp.~4869--4880.

\bibitem{higham2008functions}
{\sc N.~J. Higham}, {\em Functions of matrices: theory and computation}, SIAM, 2008.

\bibitem{himmelblau2018applied}
{\sc D.~M. Himmelblau et~al.}, {\em Applied nonlinear programming}, McGraw-Hill, 2018.

\bibitem{hopf1950partial}
{\sc E.~Hopf}, {\em The partial differential equation $u_t+uu_x=\mu_{xx}$}, Communications on Pure and Applied mathematics, 3 (1950), pp.~201--230.

\bibitem{hormander1995symplectic}
{\sc L.~H{\"o}rmander}, {\em Symplectic classification of quadratic forms, and general {M}ehler formulas}, Mathematische Zeitschrift, 219 (1995), pp.~413--449.

\bibitem{jamison1974reciprocal}
{\sc B.~Jamison}, {\em Reciprocal processes}, Zeitschrift f{\"u}r Wahrscheinlichkeitstheorie und Verwandte Gebiete, 30 (1974), pp.~65--86.

\bibitem{karatzas2012brownian}
{\sc I.~Karatzas and S.~Shreve}, {\em Brownian motion and stochastic calculus}, vol.~113, Springer Science \& Business Media, 2012.

\bibitem{kibble1945extension}
{\sc W.~Kibble}, {\em An extension of a theorem of {M}ehler's on {H}ermite polynomials}, in Mathematical Proceedings of the Cambridge Philosophical Society, vol.~41, Cambridge University Press, 1945, pp.~12--15.

\bibitem{lee2024disco}
{\sc D.~Lee, D.~Lee, D.~Bang, and S.~Kim}, {\em Disco: Diffusion {S}chr{\"o}dinger bridge for molecular conformer optimization}, in Proceedings of the AAAI Conference on Artificial Intelligence, vol.~38, 2024, pp.~13365--13373.

\bibitem{leonard2011stochastic}
{\sc C.~L{\'e}onard}, {\em Stochastic derivatives and generalized h-transforms of {M}arkov processes}, arXiv preprint arXiv:1102.3172,  (2011).

\bibitem{leonard2012schrodinger}
{\sc C.~L{\'e}onard}, {\em From the {S}chr{\"o}dinger problem to the {M}onge--{K}antorovich problem}, Journal of Functional Analysis, 262 (2012), pp.~1879--1920.

\bibitem{liu2022deep}
{\sc G.-H. Liu, T.~Chen, O.~So, and E.~Theodorou}, {\em Deep generalized {S}chr{\"o}dinger bridge}, Advances in Neural Information Processing Systems, 35 (2022), pp.~9374--9388.

\bibitem{liu2023i2sb}
{\sc G.-H. Liu, A.~Vahdat, D.-A. Huang, E.~A. Theodorou, W.~Nie, and A.~Anandkumar}, {\em I\textsuperscript{2}{SB}: {I}mage-to-{I}mage {S}chr{\"o}dinger {B}ridge}, in Proceedings of the 40th International Conference on Machine Learning, 2023, pp.~22042--22062.

\bibitem{louck1981extension}
{\sc J.~Louck}, {\em Extension of the {Kibble-Slepian} formula for {H}ermite polynomials using boson operator methods}, Advances in Applied Mathematics, 2 (1981), pp.~239--249.

\bibitem{mehler1866ueber}
{\sc F.~G. Mehler}, {\em Ueber die entwicklung einer function von beliebig vielen variablen nach laplaceschen functionen h{\"o}herer ordnung.},  (1866).

\bibitem{nagasawa2012schrodinger}
{\sc M.~Nagasawa}, {\em Schr{\"o}dinger equations and diffusion theory}, vol.~86, Birkh{\"a}user, 2012.

\bibitem{nodozi2022schrodinger}
{\sc I.~Nodozi and A.~Halder}, {\em Schr{\"o}dinger meets {K}uramoto via {Feynman-Kac}: Minimum effort distribution steering for noisy nonuniform {K}uramoto oscillators}, in 2022 IEEE 61st Conference on Decision and Control (CDC), IEEE, 2022, pp.~2953--2960.

\bibitem{nodozi2023physics}
{\sc I.~Nodozi, J.~O’Leary, A.~Mesbah, and A.~Halder}, {\em A physics-informed deep learning approach for minimum effort stochastic control of colloidal self-assembly}, in 2023 American Control Conference (ACC), IEEE, 2023, pp.~609--615.

\bibitem{nodozi2023neural}
{\sc I.~Nodozi, C.~Yan, M.~Khare, A.~Halder, and A.~Mesbah}, {\em Neural {S}chr{\"o}dinger bridge with {S}inkhorn losses: Application to data-driven minimum effort control of colloidal self-assembly}, IEEE Transactions on Control Systems Technology, 32 (2024), pp.~960--973.

\bibitem{oksendal2013stochastic}
{\sc B.~Oksendal}, {\em Stochastic differential equations: an introduction with applications}, Springer Science \& Business Media, 2013.

\bibitem{peluchetti2023diffusion}
{\sc S.~Peluchetti}, {\em Diffusion bridge mixture transports, {S}chr{\"o}dinger bridge problems and generative modeling}, Journal of Machine Learning Research, 24 (2023), pp.~1--51.

\bibitem{pravda2018generalized}
{\sc K.~Pravda-Starov}, {\em Generalized {M}ehler formula for time-dependent non-selfadjoint quadratic operators and propagation of singularities}, Mathematische Annalen, 372 (2018), pp.~1335--1382.

\bibitem{robert2021coherent}
{\sc D.~Robert and M.~Combescure}, {\em Coherent states and applications in mathematical physics}, Springer, 2021.

\bibitem{sakurai2020modern}
{\sc J.~J. Sakurai and J.~Napolitano}, {\em Modern quantum mechanics}, Cambridge University Press, 2020.

\bibitem{sanov1957probability}
{\sc I.~N. Sanov}, {\em On the probability of large deviations of random magnitudes}, Matematicheskii Sbornik, 84 (1957), pp.~11--44.

\bibitem{schrodinger1931umkehrung}
{\sc E.~Schr{\"o}dinger}, {\em {\"U}ber die umkehrung der naturgesetze}, Sitzungsberichte der Preuss. Phys. Math. Klasse, 10 (1931), pp.~144--153.

\bibitem{schrodinger1932theorie}
{\sc E.~Schr{\"o}dinger}, {\em Sur la th{\'e}orie relativiste de l'{\'e}lectron et l'interpr{\'e}tation de la m{\'e}canique quantique}, in Annales de l'institut Henri Poincar{\'e}, vol.~2, 1932, pp.~269--310.

\bibitem{shi2024diffusion}
{\sc Y.~Shi, V.~De~Bortoli, A.~Campbell, and A.~Doucet}, {\em Diffusion {S}chr{\"o}dinger bridge matching}, Advances in Neural Information Processing Systems, 36 (2024).

\bibitem{slepian1972symmetrized}
{\sc D.~Slepian}, {\em On the symmetrized kronecker power of a matrix and extensions of {M}ehler’s formula for hermite polynomials}, SIAM Journal on Mathematical Analysis, 3 (1972), pp.~606--616.

\bibitem{song2020score}
{\sc Y.~Song, J.~Sohl-Dickstein, D.~P. Kingma, A.~Kumar, S.~Ermon, and B.~Poole}, {\em Score-based generative modeling through stochastic differential equations}, arXiv preprint arXiv:2011.13456,  (2020).

\bibitem{teter2023contraction}
{\sc A.~M. Teter, Y.~Chen, and A.~Halder}, {\em On the contraction coefficient of the {S}chr{\"o}dinger bridge for stochastic linear systems}, IEEE Control Systems Letters, 7 (2023), pp.~3325--3330.

\bibitem{teter2024solution}
{\sc A.~M. Teter, I.~Nodozi, and A.~Halder}, {\em Solution of the probabilistic {L}ambert problem: Connections with optimal mass transport, {S}chr\"{o}dinger bridge and reaction-diffusion {PDE}s}, arXiv preprint arXiv:2401.07961,  (2024).

\bibitem{vargas2023bayesian}
{\sc F.~Vargas, A.~Ovsianas, D.~Fernandes, M.~Girolami, N.~D. Lawrence, and N.~N{\"u}sken}, {\em Bayesian learning via neural {S}chr{\"o}dinger--{F}{\"o}llmer flows}, Statistics and Computing, 33 (2023), p.~3.

\bibitem{vargas2024transport}
{\sc F.~Vargas, S.~Padhy, D.~Blessing, and N.~N{\"u}sken}, {\em Transport meets variational inference: Controlled {M}onte {C}arlo diffusions}, in The Twelfth International Conference on Learning Representations, 2024.

\bibitem{wakolbinger1990schrodinger}
{\sc A.~Wakolbinger}, {\em Schr{\"o}dinger bridges from 1931 to 1991}, in Proc. of the 4th Latin American Congress in Probability and Mathematical Statistics, Mexico City, 1990, pp.~61--79.

\bibitem{wang2021deep}
{\sc G.~Wang, Y.~Jiao, Q.~Xu, Y.~Wang, and C.~Yang}, {\em Deep generative learning via {S}chr{\"o}dinger bridge}, in International conference on machine learning, PMLR, 2021, pp.~10794--10804.

\bibitem{zee2010quantum}
{\sc A.~Zee}, {\em Quantum field theory in a nutshell}, vol.~7, Princeton university press, 2010.

\bibitem{zhou2023denoising}
{\sc L.~Zhou, A.~Lou, S.~Khanna, and S.~Ermon}, {\em Denoising diffusion bridge models}, arXiv preprint arXiv:2309.16948,  (2023).

\bibitem{zinn2021quantum}
{\sc J.~Zinn-Justin}, {\em Quantum field theory and critical phenomena}, vol.~171, Oxford university press, 2021.

\end{thebibliography}
